\documentclass[twoside]{article}

\usepackage{PRIMEarxiv}

\usepackage[utf8]{inputenc} 
\usepackage[T1]{fontenc}    
\usepackage{hyperref}       
\usepackage{url}            
\usepackage{booktabs}       
\usepackage{amsfonts}       
\usepackage{nicefrac}       
\usepackage{microtype}      
\usepackage{lipsum}
\usepackage{mathtools}
\usepackage{multicol}
\usepackage{fancyhdr}       
\usepackage{graphicx}       
\graphicspath{{media/}}     
\usepackage{relsize}
\usepackage{amsmath}
\usepackage{amsthm}
\usepackage{amssymb}
\usepackage{tikz}
\usepackage{tikz-cd}
\usetikzlibrary{calc}
\usepackage{faktor}
\usepackage{adjustbox}
\usepackage{yhmath}
\usepackage{scalerel}

\usetikzlibrary{arrows,calc,matrix}

  \newtheorem{theorem}{Theorem}[section]
	\newtheorem{proposition}[theorem]{Proposition}
	
	\newtheorem{lemma}[theorem]{Lemma}
	\newtheorem{corollary}[theorem]{Corollary}
	\theoremstyle{definition}
	\newtheorem*{claim*}{Claim}
	\newtheorem{remark}[theorem]{Remark}
	
	\newtheorem{definition}[theorem]{Definition}
	
	\newtheorem{example}[theorem]{Example}
	
	\theoremstyle{remark}

   \DeclareMathOperator {\lk}{lk} 
    \DeclareMathOperator {\st}{st}

  \DeclareMathOperator {\Out}{Out} 
  \DeclareMathOperator {\Aut}{Aut} 

\usepackage{flowchart}
\usetikzlibrary{arrows}

\newcommand{\abs}[1]{\left|#1\right|}
\usepackage{scalerel,stackengine}
\stackMath
\newcommand\reallywidehat[1]{%
	\savestack{\tmpbox}{\stretchto{%
			\scaleto{%
				\scalerel*[\widthof{\ensuremath{#1}}]{\kern-.6pt\bigwedge\kern-.6pt}%
				{\rule[-\textheight/2]{1ex}{\textheight}}
			}{\textheight}%
		}{0.5ex}}%
	\stackon[1pt]{#1}{\tmpbox}%
}
\newcommand\restr[2]{{
		\left.\kern-\nulldelimiterspace 
		#1 
		\right|_{#2} 
}}
  
\title{$\ell^2$-Betti numbers and $\Sigma$-invariants of graph products of cyclic groups and their automorphisms}
\author{Marcos Escartín-Ferrer}

\begin{document}
 \maketitle

\begin{abstract}
	We study the $\ell^2$-Betti numbers of graph products of cyclic groups. We relate these invariants to the $\ell^2$-Betti numbers of right-angled Artin groups and obtain explicit formulas in several cases determined by the structure of the finite-order vertices in the defining graph. In particular, we consider the cases where these vertices form a clique or have identical links, and derive a formula for the first $\ell^2$-Betti number when the defining graph has exactly two finite-order vertices. We also compute the BNSR-invariants of graph products of cyclic groups and the first $\ell^2$-Betti number of their automorphism groups when the defining graph has at most two finite-order vertices.
\end{abstract}
\section*{Acknowledgements}
The author is partially supported by the Departamento de Ciencia, Universidad y Sociedad del Conocimiento del Gobierno de Aragón (grant code: E22-23R: “Álgebra y Geometría”), and by the Spanish Government PID2021-126254NBI00 and PID2024155800NB-C32.The author would like to thank Dawid Kielak for helpful conversations during a short stay in Oxford.
\section{Introduction}
Let $\Gamma$ be a finite simplicial graph together with a map $o: V(\Gamma)\to \mathbb{Z}_{\geq 2}\cup\lbrace\infty\rbrace$. The \textbf{graph product of cyclic groups} $\mathcal{G}_\Gamma$ is defined by the presentation
$$\mathcal{G}_\Gamma=\left\langle V(\Gamma)\middle|,
uv=vu\ \text{for every } \{u,v\}\in E(\Gamma),
v^{o(v)}=1\ \text{if }o(v)<\infty
\right\rangle.
$$
When $o(v)=\infty$ for every $v\in V(\Gamma)$, the group $\mathcal{G}_\Gamma$ is the right-angled Artin group (RAAG) $A_\Gamma$. When $o(v)=2$ for every $v\in V(\Gamma)$, it is the right-angled Coxeter group (RACG) $W_\Gamma$. Thus, graph products of cyclic groups provide a common framework containing both RAAGs and RACGs, while allowing the vertex groups to have arbitrary finite orders.

The interplay between the algebraic structure of a graph product and the combinatorics of its defining graph has been a central theme in the study of these groups. Many group-theoretic and homological properties of graph products can be described in terms of links, stars, induced subgraphs, and the distribution of finite and infinite  order vertices. In particular, the presence of finite order vertices introduces phenomena that do not occur for right-angled Artin groups and leads to additional contributions to various invariants. In this paper, we investigate this interaction through the $\ell^2$-Betti numbers, BNSR-invariants, and automorphism groups of graph product of cyclic groups.

We first study the $\ell^2$-Betti numbers of graph products of cyclic groups. Davis and Leary computed the $\ell^2$-Betti numbers of RAAGs. More precisely, if $A_\Gamma$ is a RAAG, then $\beta_{i+1}^{(2)}(A_\Gamma)=\overline{\beta}_i(\widehat{\Gamma})$ for every $i\geq 0$, where $\overline{\beta}_i$ denotes the $i$-th reduced Betti number and $\widehat{\Gamma}$ is the \textbf{flag complex} of $\Gamma$, defined as the simplicial complex obtained by filling in an $(n-1)$-simplex for every $n$-clique in $\Gamma$ \cite{Davis}.

For RACGs, although there are many partial results (cf. \cite{Avramidi, Davis Okun, Davis Okun 2,Jankiewicz}), a general formula for their $\ell^2$-Betti numbers remains an open problem.

These two cases represent the two extreme cases for the orders of the vertex groups. We focus instead on the intermediate setting, where both finite and infinite order vertices occur in the defining graph. Denote by $\Gamma^{<\infty}$ the subgraph induced by the finite-order vertices, and suppose that $V(\Gamma^{<\infty})=\{u_1,\dots,u_n\}$. We can then define the following homomorphism:
$$\varphi:\mathcal{G}_\Gamma\longrightarrow\prod_{i=1}^n\mathbb{Z}/o(u_i)\mathbb{Z},$$
where $\varphi(u_i)=(0,\dots,0,1,0,\dots,0)$, with the $1$ in the $i$-th component, and $\varphi(v)=0$ whenever $o(v)=\infty$. The kernel $\ker(\varphi)$ is a finite-index subgroup of $\mathcal{G}_\Gamma$, for which an explicit presentation can be obtained using the Reidemeister--Schreier algorithm. By studying this kernel, we identify two cases in which it is isomorphic to a RAAG, allowing us to compute the $\ell^2$-Betti numbers explicitly in terms of the $\ell^2$-Betti numbers of RAAGs associated to certain subgraphs of $\Gamma$.
\begin{theorem}\label{Main1}
	Let $\mathcal{G}_\Gamma$ be a graph product of cyclic groups. Then:
	\begin{enumerate}
		\item If $\Gamma^{<\infty}$ is a clique, then
		$$ \beta_i^{(2)}(\mathcal{G}_\Gamma)=\frac{1}{\abs{\mathcal{G}_{\Gamma^{<\infty}}}}\left(\sum_{\Omega\subset\Gamma^{<\infty}}\left(\prod_{w\in V(\Gamma^{<\infty}\setminus\Omega)}(o(w)-1)\right)\beta_i^{(2)}(A_{\Gamma\setminus\Omega})\right)\qquad\text{for all } i\geq 0.$$
		\item If $V(\Gamma^{<\infty})=\{u_1,\dots,u_n\}$ and $\lk(u_1)=\cdots=\lk(u_n)=:L$,	then
		$$
		\beta_i^{(2)}(\mathcal{G}_\Gamma)
		=
		\beta_i^{(2)}(A_{\Gamma_{\mathcal{U}_1}})
		+
		\left(
		n-1-\sum_{j=1}^{n}\frac{1}{o(u_j)}
		\right)
		\beta_{i-1}^{(2)}(A_L)
		\qquad\text{for all } i\geq 0,
		$$
		where $\Gamma_{\mathcal{U}_1}$ is the graph obtained from $\Gamma$ by deleting $u_2,\dots,u_n$.
	\end{enumerate}
\end{theorem}
Here, given a subgraph $\Delta\subset\Gamma$, we denote by $A_\Delta$ the RAAG associated to the graph $\Delta$, that is, the group obtained from $\mathcal{G}_\Delta$ by removing the finite-order relators.

For a torsion-free group $G$ satisfying the Atiyah conjecture, Lück showed in \cite[Lemma 10.28 (3)]{Luck} that the $\ell^2$-Betti numbers of $G$ can be expressed as $\beta_i^{(2)}(G)=\dim_{\mathcal{D}(G)} \mathrm{H}^i(G,\mathcal{D}(G))$ for all $i\geq0$, where $\mathcal{D}(G)$ denotes the Linnell division ring of $G$. By studying the space of derivations $\mathrm{Der}(G,\mathcal{D}(G))$ and using the identification

$$\mathrm{H}^1(G,\mathcal{D}(G))
\cong
\frac{\mathrm{Der}(G,\mathcal{D}(G))}
{\mathrm{InnDer}(G,\mathcal{D}(G))},$$

we compute the first $\ell^2$-Betti number of of graph products of cyclic groups with two finite-order vertices.
\begin{theorem}\label{Main2}
	Let $\mathcal{G}_\Gamma$ be a graph product of cyclic groups with $V(\Gamma^{<\infty})=\{u,v\}$. Then
	
		$$\beta_1^{(2)}(\mathcal{G}_\Gamma)=\frac{1}{o(u)o(v)}\left((o(u)-1)(o(v)-1)\beta_1^{(2)}(A_\Gamma)+(o(v)-1)\beta_{1}^{(2)}(A_{\Gamma_u})+(o(u)-1)\beta_{1}^{(2)}(A_{\Gamma_v})+\beta_{1}^{(2)}(A_{\Gamma_{uv}})\right).$$
\end{theorem}
Using this approach, we also give an alternative proof of the computation of the $\ell^2$-Betti numbers of torsion-free one-relator groups (see \cite{Dicks} for the original proof) and compute the first $\ell^2$-Betti number of a locally indicable twisted RAAG (see Section~\ref{Twisted RAAGs} for the definition).

A second aspect of the paper concerns the BNSR-invariants of graph products of cyclic groups. Given a group $G$, the BNSR-invariants are subsets of the character sphere $S(G)$, whose elements are equivalence classes of nonzero \textbf{characters}, that is, homomorphisms $\chi:G\to\mathbb{R}$. There are two families of invariants: one of homotopical flavour and one of homological flavour, denoted by $\Sigma^n(G)$ and $\Sigma^n(G,A)$, respectively, where $A$ is a left $G$-module. A key feature of these invariants is that they encode a great deal of information about the $F_n$ and $FP_n$ properties of \textbf{coabelian subgroups} of $G$, that is, subgroups $H\leq G$ such that $G/H$ is abelian.

A group $G$ is said to be \textbf{algebraically fibred} if it admits an epimorphism
$G\twoheadrightarrow\mathbb{Z}$ with finitely generated kernel. More generally, a group $G$ is said to be \textbf{$\mathcal{P}$-algebraically fibred}, where $\mathcal{P}$ is a finiteness property such as $F_n$ or $FP_n$, if it admits an epimorphism
$G\twoheadrightarrow\mathbb{Z}$ whose kernel is of type $\mathcal{P}$.

For graph products of cyclic groups, given an epimorphism
$\chi:G\twoheadrightarrow\mathbb{Z}$, it follows that
$[\chi]\in\Sigma^n(G)$ if and only if $\ker(\chi)$ is of type $F_n$, and
$[\chi]\in\Sigma^n(G,\mathbb{Z})$ if and only if $\ker(\chi)$ is of type $FP_n$. Kielak and Fisher proved in \cite{Fisher, Kielak} that for a virtually RFRS group $G$ virtual fibring of type $FP_n$ is equivalent to the vanishing of the $\ell^2$-Betti numbers $\beta_i^{(2)}(G)=0$ for all $i=0,\dots,n$, and graph products of cyclic groups are virtually RFRS (cf.~\cite{Huang}), this provides a connection between the $\ell^2$-Betti numbers and the BNSR-invariants.

For a RAAG $A_\Gamma$, Meier, Meinert, and VanWyk gave in \cite{Meier-Meinert-VanWyk 1998} combinatorial criteria, in terms of $\Gamma$, for determining whether a character $\chi:A_\Gamma\to\mathbb{R}$ lies in $\Sigma^n(A_\Gamma)$ or $\Sigma^n(A_\Gamma,\mathbb{Z})$. The corresponding conditions are referred to as the \textbf{homotopical $n$-link condition} and the \textbf{homological $n$-link condition}, respectively. We refer the reader to Section~\ref{SigmaSection} for their definitions. We prove that the homological $n$-link condition gives a complete characterization in the homological setting and that it is sufficient for membership in the homotopical invariant.

\begin{theorem}\label{Main3}
	Let $\mathcal{G}_\Gamma$ be a graph product of cyclic groups and let $\chi:\mathcal{G}_\Gamma\to\mathbb{R}$ be a non-zero character. Then:
	\begin{enumerate}
		\item $[\chi]\in\Sigma^n(\mathcal{G}_\Gamma,\mathbb{Z})$ if and only if $\chi$ satisfies the homological $n$-link condition.
		\item If $ [\chi]\in\Sigma^n(\mathcal{G}_\Gamma)$, then $\chi$ satisfies the homotopical $n$-link condition.
	\end{enumerate}
\end{theorem}

For the converse implication in the homotopical setting, we obtain a weaker statement. Namely, we prove that if $\chi$ satisfies a stronger condition, which we call the \textbf{strong homological $n$-link condition}, then  $[\chi]\in\Sigma^n(\mathcal{G}_\Gamma)$.

Finally, we consider automorphism groups of graph products of cyclic groups. In \cite{Escartin}, the author computed the first $\ell^2$-Betti number of a RAAG and proved that

$$\beta_1^{(2)}(\operatorname{Aut}(A_\Gamma))\neq 0\quad\text{if and only if}\quad A_\Gamma\cong\mathbb{Z}^2.$$

Using the formulas obtained above for the $\ell^2$-Betti numbers of graph products, we compute the first $\ell^2$-Betti number of the automorphism group of a graph product of cyclic groups with exactly one finite-order vertex.

\begin{theorem}\label{Main4}
	Let $\mathcal{G}_\Gamma$ be a graph product of cyclic groups with
	$V(\Gamma^{<\infty})=\{v\}$. Then
	$$\beta_1^{(2)}(\operatorname{Aut}(\mathcal{G}_\Gamma))\neq 0
	\quad \text{if and only if} \quad
	\mathcal{G}_\Gamma \in
	\bigl\{
	\mathbb{Z}^2\times \mathbb{Z}/o(v)\mathbb{Z},
	\mathbb{Z}\ast \mathbb{Z}/o(v)\mathbb{Z}
	\bigr\}.$$
\end{theorem}
We also compute the first $\ell^2$-Betti number of the automorphism group of a graph product of cyclic groups with two finite-order vertices. However, the resulting description is more involved, and we therefore omit it from the introduction; see Theorem~\ref{TheoremAut2}.

The paper is organized as follows. In Section~2, we recall the necessary background on graph products and their automorphisms, and $\Sigma$-invariants. Section~3 is devoted to the study of $\ker(\varphi)$ and the proof of Theorem~\ref{Main1}. In Section~4, we study the first $\ell^2$-Betti number using derivations and prove Theorem~\ref{Main2}. Section~5 is devoted to the study of the BNSR-invariants, where we prove Theorem~\ref{Main3}. Finally, in Section~6, we study automorphism groups and prove Theorem~\ref{Main4}.
\section{Preliminaries}
\subsection{$\Sigma$-Invariants}
Define the \textbf{character sphere} of a finitely generated group $G$ as
$$
S(G)=(\mathrm{Hom}(G,\mathbb{R})\setminus\{0\})/\sim
$$
where $\chi_1,\chi_2\in\mathrm{Hom}(G,\mathbb{R})\setminus\lbrace 0\rbrace$ are equivalent if $\chi_1=t\chi_2$ for some $t>0$. If $n$ denotes the torsion-free rank of the abelianization of $G$, then $S(G)\cong\mathbb{S}^{n-1}$ (c.f. \cite[Lemma A1.1]{Strebel}). A character $\chi:G\to \mathbb{Z}$ is called \textbf{discrete}, and the set of all discrete characters is dense in $S(G)$ (c.f. \cite[Lemma B3.24]{Strebel}). We define the homological and homotopical Sigma invariants of a group as follows:

\begin{definition}
	Let $G$ be a finitely generated group and $\chi:G\to\mathbb{R}$ a character. If $A$ is a left $G$-module and $n\geq 1$, the \textbf{Bieri–Neumann–Strebel–Renz (BNSR) invariants} are defined by
	$$
	\Sigma^n(G,A)=\{[\chi]\in S(G)\mid A \text{ is of type }\mathrm{FP}_n \text{ over }\mathbb{Z}G_\chi\}
	$$
	where $G_\chi=\{g\in G\mid\chi(g)\geq 0\}$.
\end{definition}

\begin{definition}
	Let $G$ be a group admitting a finite $K(G,1)$, and let $X$ be a finite model for $K(G,1)$ having a single 0-cell. There is a bijection between the vertices of the universal cover $\widetilde{X}$ and the elements of $G$. Fix a base point $b\in\widetilde{X}^{(0)}$. Given a character $\chi:G\to\mathbb{R}$, define a function $\widetilde{\chi}:\widetilde{X}^{(0)}\to\mathbb{R}$ by $\widetilde{\chi}(g\cdot b)=\chi(g)$. For $a\in\mathbb{R}$, let $\widetilde{X_{\chi}}^{[a,\infty)}$ be the maximal subcomplex of $\widetilde{X}$ with $0$-skeleton $\{x\in\widetilde{X}^{(0)}\mid \widetilde{\chi}(x)\geq a\}$. Then, for each $d\geq 0$, the inclusion $\widetilde{X_{\chi}}^{[0,\infty)}\hookrightarrow \widetilde{X_{\chi}}^{[-d,\infty)}$ induces a map $\tau_i^d:\pi_i\bigl(\widetilde{X_{\chi}}^{[0,\infty)}\bigr)\longrightarrow \pi_i\bigl(\widetilde{X_{\chi}}^{[-d,\infty)}\bigr).$
	Define the \textbf{homotopical Sigma invariants} by
	$$\Sigma^n(G)=\Bigl\{[\chi]\in S(G)\ \Bigm|\ \exists\, d>0 \text{ such that }\tau_i^d\text{ is trivial for all } i<n \Bigr\}.$$
\end{definition}
Observe that:
	$$\Sigma^\infty(G,A)\subseteq \dots \subseteq \Sigma^n(G,A)\subseteq \dots \subseteq \Sigma^1(G,A)\subseteq \Sigma^0(G,A)=S(G).$$
$$\Sigma^\infty(G)\subseteq \dots \subseteq \Sigma^1(G)\subseteq \Sigma^0(G)=S(G).$$

Renz proved in his thesis (cf. \cite{Renz}) that the homotopical Sigma invariants are well-defined, i.e. they do not depend on the chosen model for $K(G,1)$. Bieri and Renz also showed that both $\Sigma^n(G,\mathbb{Z})$ and $\Sigma^n(G)$ are open in $S(G)$, and that if $n=\mathrm{cd}(G)$ is the cohomological dimension of $G$, then $\Sigma^m(G,\mathbb{Z})=\Sigma^n(G,\mathbb{Z})$ for every $m\geq n$. Similarly, if $n=\mathrm{gd}(G)$ is the geometric dimension of $G$, then $\Sigma^m(G)=\Sigma^n(G)$ for every $m\geq n$ (c.f.\ \cite{Bieri-Renz}).

Another important remark is that the homotopical and homological Sigma invariants are closely related. Indeed, $\Sigma^1(G)=\Sigma^1(G,\mathbb{Z})$ and $\Sigma^n(G)=\Sigma^n(G,\mathbb{Z})\cap\Sigma^2(G)$ for every $n\geq 1$ (c.f.\ \cite{Renz}). The main result about Sigma invariants, which explains why they encode the finiteness properties of the coabelian subgroups, is the following:
\begin{theorem}[Bieri-Renz \cite{Bieri-Renz}]\label{teo:FinitenesSigma}
	Let $G$ be a group of type $\mathrm{F}_n$ and $G'\leq N\leq G$. Then:
	\begin{enumerate}
		\item[(i)] $N$ is of type $\mathrm{FP}_n$ if and only if $\{[\chi]\in S(G)\mid \chi(N)=0\}\subseteq \Sigma^n(G,\mathbb{Z})$.
		\item[(ii)] $N$ is of type $\mathrm{F}_n$ if and only if
		$\{[\chi]\in S(G)\mid \chi(N)=0\}\subseteq \Sigma^n(G)$.
	\end{enumerate}
\end{theorem}
The next Theorem, which will be the key to compute the $\Sigma$-invariants of graph products of cyclic groups, is due to Schmitt and Meinert. 
\begin{theorem}[Theorem 3.2 \cite{Meier-Meinert-VanWyk 1998}]\label{teo:MMVW}
	Let $X$ be a $G$–CW–complex with cocompact $n$-skeleton, and let $\chi:G\to\mathbb{R}$ be a character such that for every $p$-cell $\sigma$ of $X$ with $p\leq n$, we have $G_\sigma\nsubseteq\ker(\chi)$, where $G_\sigma$ denotes the stabilizer of $\sigma$. Then:
	\begin{itemize}
		\item[(i)] If $X$ is $(n-1)$-connected and $[\chi|_{G_\sigma}]\in\Sigma^{\,n-p}(G_\sigma)$ for each $p$-cell $\sigma$ with $p\leq n$, then $[\chi]\in\Sigma^n(G)$.
		
		\item[(ii)] If $X$ is $(n-1)$-acyclic and $[\chi|_{G_\sigma}]\in\Sigma^{\,n-p}(G_\sigma,\mathbb{Z})$ for each $p$-cell $\sigma$ with $p\leq n$, then $[\chi]\in\Sigma^n(G,\mathbb{Z})$.
	\end{itemize}
\end{theorem}
\subsection{Graph product of cyclic groups and their automorphisms}
We begin by recalling several well-known structural properties of graph products of cyclic groups.

\begin{lemma}[\cite{Green}]
	Let $\mathcal{G}_\Gamma$ be a graph product of cyclic groups. Then:
	\begin{enumerate}
		\item Let $\Delta$ be an \textbf{induced subgraph} of $\Gamma$; that is, a subgraph such that whenever $v,w\in V(\Delta)$ satisfy ${v,w}\in E(\Gamma)$, then ${v,w}\in E(\Delta)$. The subgroup of $\mathcal{G}_\Gamma$ generated by $V(\Delta)$ is isomorphic to $\mathcal{G}_\Delta$. Such subgroups are called \textbf{special subgroups}.
		\item The centre $\mathrm{Z}(\mathcal{G}_\Gamma)$ of $\mathcal{G}_\Gamma$ is the special subgroup $\mathcal{G}_\Delta$, where $\Delta$ is the subgraph induced by the vertices $v\in V(\Gamma)$ satisfying $\mathrm{st}(v)=\Gamma$.
		\item A special subgroup $\mathcal{G}_\Delta$ is finite if and only if $\Delta$ is complete and each vertex group has finite order.
	\end{enumerate}
\end{lemma}
There is one induced subgraph that will be of particular importance in this paper, namely $\Gamma^{<\infty}$, the subgraph induced by the finite-order vertices of $\mathcal{G}_\Gamma$. In the rest of the paper, given a set of vertices ${v_1,\dots,v_n}$ of $\Gamma$, we denote by $\Gamma_{v_1\cdots v_n}$ the subgraph of $\Gamma$ obtained by deleting $v_1,\dots,v_n$.

We will make use of the known formulas for the $\ell^2$-Betti numbers of RAAGs to study the $\ell^2$-Betti numbers of graph products of cyclic groups.
\begin{theorem}[Davis-Leary \cite{Davis}]\label{BettiRAAG} Let $A_\Gamma$ be a RAAG. Then:
	$$\beta_{i+1}^{(2)}(A_\Gamma)=\overline{\beta}_i(\widehat{\Gamma})~\mathrm{for~every~}i\geq 0$$
where $\overline{\beta}_i$ denote the reduced  $i$-th Betti number and  $\widehat{\Gamma}$ is the \textbf{flag complex} of $\Gamma$, defined as the simplicial complex obtained by filling in an $(n-1)$-simplex for every $n$-clique in $\Gamma$.
\end{theorem}

To finish the section let us recall that there exists a canonical generating set for the automorphism group of a graph product of cyclic groups.
\begin{theorem}[Corredor-Gutierrez \cite{Corredor-Gutierrez}]\label{GeneratorsAUT}
	Let $\mathcal{G}_\Gamma$ be a graph product of cyclic groups. Then $\mathrm{Aut}(\mathcal{G}_\Gamma)$ is generated by the following types of automorphisms:
	\begin{itemize}
		\item \textbf{Labelled graph automorphisms:} every automorphism of $\Gamma$ preserving vertex labels induces an automorphism of $\mathcal{G}_\Gamma$.
		\item \textbf{Factor automorphisms:} every automorphism of an individual vertex group extends to an automorphism of $\mathcal{G}_\Gamma$.
		\item \textbf{Transvections:} for $u,v\in V(\Gamma)$, define $\tau^v_u$ as follows:
		\begin{itemize}
			\item If $\mathrm{lk}(u)\subset\mathrm{st}(v)$ and $o(u)=\infty$, set $\tau^v_u(u)=uv$ and fix all other generators.
			\item If $\mathrm{st}(u)\subset\mathrm{st}(v)$ and there exists a prime $p$ such that $o(u)=p^i$ and $o(v)=p^j$, define $\tau^v_u(u)=uv^k$, where $k=\max\{1,p^{j-i}\}$.
		\end{itemize}
		\item \textbf{Partial conjugations:} if $v\in V(\Gamma)$ and $\Gamma\setminus\mathrm{st}(v)\neq\emptyset$, and $C$ is a connected component of $\Gamma\setminus\mathrm{st}(v)$, define $\pi^v_C$ by
		$$\pi^v_C(w)=
		\begin{cases}
			vwv^{-1}, & \text{if } w\in V(C),\\[2pt]
			w, & \text{otherwise.}
		\end{cases}$$
	\end{itemize}
\end{theorem}
\section{$\ell^2$-Betti numbers of graph products of cyclic groups}
\subsection{A finite index subgroup of graph products of cyclic groups}
Let $\mathcal{G}_\Gamma$ be a graph product of cyclic groups and let $\mathcal{U}=\{u_1,\dots,u_n\}$ be a set of finite order vertices of $\Gamma$. Define a group homomorphism
$$
\varphi:\mathcal{G}_\Gamma\to\prod_{i=1}^n\mathbb{Z}/o(u_i)\mathbb{Z}
$$
by $\varphi(u_i)=(0,\dots,0,1,0,\dots,0)$, where the $1$ is in the $i$-th component, and $\varphi(v)=0$ if $v\notin \mathcal{U}$. We compute a presentation for $\ker(\varphi)$ using the Reidemeister-Schreier algorithm.
By the First Isomorphism Theorem,
$$
\mathcal{G}_\Gamma/\ker(\varphi)\cong\mathrm{im}(\varphi)\cong
\prod_{i=1}^n\mathbb{Z}/o(u_i)\mathbb{Z},
$$
so $\ker(\varphi)$ is a subgroup of index $m:=\prod\limits_{i=1}^n o(u_i)$. Consider the following Schreier transversal:
$$
T=\{u_1^{i_1}\cdots u_n^{i_n}\mid 0\le i_j<o(u_j)\text{ for }1\le j\le n\}.
$$

We first compute the generators of $\ker(\varphi)$:

\begin{itemize}
	\item If $v\notin \mathcal{U}$, then $\ker(\varphi)u_1^{i_1}\cdots u_n^{i_n}v
	=\ker(\varphi)u_1^{i_1}\cdots u_n^{i_n}$,
	so there exists $v_{i_1,\dots,i_n}\in\ker(\varphi)$ such that $v_{i_1,\dots,i_n}u_1^{i_1}\cdots u_n^{i_n}
	=u_1^{i_1}\cdots u_n^{i_n}v$.
	Thus,
	$$
	v_{i_1,\dots,i_n}
	=u_1^{i_1}\cdots u_n^{i_n}v(u_1^{i_1}\cdots u_n^{i_n})^{-1}
	=v^{u_1^{i_1}\cdots u_n^{i_n}}.
	$$
	
	\item For $u_j$, we have $\ker(\varphi)u_1^{i_1}\cdots u_n^{i_n}u_j
	=\ker(\varphi)u_1^{i_1}\cdots u_j^{i_j+1}\cdots u_n^{i_n}$, so there exists $\prescript{j}{}{X_{i_1,\dots,i_n}}\in\ker(\varphi)$ such that $
	\prescript{j}{}{X_{i_1,\dots,i_n}}
	u_1^{i_1}\cdots u_j^{i_j+1}\cdots u_n^{i_n}
	=u_1^{i_1}\cdots u_n^{i_n}u_j.
	$
	Hence,
	$$
	\prescript{j}{}{X_{i_1,\dots,i_n}}
	=[u_{j+1}^{i_{j+1}}\cdots u_n^{i_n},u_j]^{u_1^{i_1}\cdots u_j^{i_j}}.
	$$
\end{itemize}
Therefore, $\ker(\varphi)$ is generated by
$$
\{v_{i_1,\dots,i_n}\mid 0\le i_j<o(u_j)\text{ for }1\le j\le n\}
\cup\{\prescript{j}{}{X_{i_1,\dots,i_n}}
\mid 1\le j\le n-1,\ 0\le i_k<o(u_k)\text{ for }1\le k\le n\}.
$$
where $\prescript{j}{}{X_{i_1,\dots,i_n}}=1$ if and only if $u_j$ commutes with $u_{j+1}^{i_{j+1}}\cdots u_n^{i_n}$. We now compute the relators of $\ker(\varphi)$.

\begin{itemize}
	\item If $v,w\notin \mathcal{U}$ and $\{v,w\}\in E(\Gamma)$, then $u_1^{i_1}\cdots u_n^{i_n}vw=u_1^{i_1}\cdots u_n^{i_n}wv$,
	which implies
	$$
	v_{i_1,\dots,i_n}w_{i_1,\dots,i_n}
	=w_{i_1,\dots,i_n}v_{i_1,\dots,i_n}.
	$$
	
	\item If $v\notin \mathcal{U}$ and $1\le j\le n$ with
	$\{v,u_j\}\in E(\Gamma)$, then $u_1^{i_1}\cdots u_n^{i_n}vu_j=u_1^{i_1}\cdots u_n^{i_n}u_jv$, which implies
	$$v_{i_1,\dots,i_n}\prescript{j}{}{X_{i_1,\dots,i_n}}
	=\prescript{j}{}{X_{i_1,\dots,i_n}}v_{i_1,\dots,i_j+1,\dots,i_n}\Rightarrow
	v_{i_1,\dots,i_n}
	=v_{i_1,\dots,i_j+1,\dots,i_n}^{\prescript{j}{}{X_{i_1,\dots,i_n}}}.
	$$

	\item If $\{u_k,u_j\}\in E(\Gamma)$, then $u_1^{i_1}\cdots u_n^{i_n}u_ku_j
	=u_1^{i_1}\cdots u_n^{i_n}u_ju_k$, which implies
	$$
	\prescript{k}{}{X_{i_1,\dots,i_n}}\prescript{j}{}{X_{i_1,\dots,i_k+1,\dots,i_n}}=\prescript{j}{}{X_{i_1,\dots,i_n}}\prescript{k}{}{X_{i_1,\dots,i_j+1,\dots,i_n}}.
	$$
	\item For each relator of the form $v^{o(v)}=1$ with $v\notin \mathcal{U}$ we obtain 
	$$u_1^{i_1}\cdots u_n^{i_n}v^{o(v)}=u_1^{i_1}\cdots u_n^{i_n}\Rightarrow v_{i_1,\dots,i_n}^{o(v)}=1$$
	\item For each relator of the form $u_j^{o(u_j)}=1$, we obtain
	$u_1^{i_1}\cdots u_n^{i_n}u_j^{o(u_j)}=u_1^{i_1}\cdots u_n^{i_n}$,
	and hence
	$$
	\prod_{k=0}^{o(u_j)-1}
	\prescript{j}{}{X_{i_1,\dots,i_j+k,\dots,i_n}}=1.
	$$
\end{itemize}

Observe that for each $u_j$, the relators of the last type are redundant when $i_j$ varies and all other indices are fixed. Hence, it suffices to consider the single relator
$$
\prod_{k=0}^{o(u_j)-1}
\prescript{j}{}{X_{i_1,\dots,k,\dots,i_n}}=1.
$$

\begin{proposition}\label{Presentation}
	$\ker(\varphi)$ is a finitely presented group with presentation given by:
	$$\stretchleftright[1000]{<}{\begin{array}{l|cl}
		v_{i_1,\dots,i_n}, \text{ for }v\notin \mathcal{U} 	\text{ and } &v_{i_1,\dots,i_n}w_{i_1,\dots,i_n}=w_{i_1,\dots,i_n}v_{i_1,\dots,i_n}\text{ if }\lbrace v,w\rbrace\in E(\Gamma) \\
		0\leq i_k<o(u_k)\text{ with }k=1,\dots,n&v_{i_1,\dots,i_n}
		=v_{i_1,\dots,i_j+1,\dots,i_n}^{\prescript{j}{}{X_{i_1,\dots,i_n}}}\text{ if }\lbrace v,u_j\rbrace\in E(\Gamma) \\
	&\prescript{k}{}{X_{i_1,\dots,i_n}}\prescript{j}{}{X_{i_1,\dots,i_k+1,\dots,i_n}}=\prescript{j}{}{X_{i_1,\dots,i_n}}\prescript{k}{}{X_{i_1,\dots,i_j+1,\dots,i_n}}\text{ if }\lbrace u_i,u_j\rbrace\in E(\Gamma)\\
	 	& v_{i_1,\dots,i_n}^{o(v)}=1\text{ if }o(v)<\infty\\
		\prescript{j}{}{X_{i_1,\dots,i_n}}\text{ for every }1\leq j\leq n-1,& \prod_{k=0}^{o(u_j)-1}
		\prescript{j}{}{X_{i_1,\dots,k,\dots,i_n}}=1\text{ for every }1\leq j\leq n-1\\
		0\leq i_k<o(u_k)\text{ with }1\leq k\leq n-1&\prescript{j}{}{X_{i_1,\dots,i_n}}=1\text{ if }u_j\text{ commutes with }u_{j+1}^{i_{j+1}}\cdots u_n^{i_n}\\
		\end{array}}{>}
	   $$
\end{proposition}

Since $\ker(\varphi)$ is a finite-index subgroup of $\mathcal{G}_\Gamma$, it suffices to understand the $\ell^2$-Betti numbers of $\ker(\varphi)$ in order to compute those of $\mathcal{G}_\Gamma$. The group $\ker(\varphi)$ is generally quite complicated, and we do not expect to understand it for an arbitrary graph $\Gamma$. Nevertheless, in the next two sections we will consider two situations in which $\ker(\varphi)$ can be analysed explicitly.

\subsection{$\Gamma^{<\infty}$ is a clique}\label{SectionClique}
We now assume that $\st(u_1)=\dots=\st(u_n)$. Under this hypothesis, each $\prescript{j}{}{X_{i_1,\dots,i_n}}$ is trivial, the subgraph $\mathcal{U}$ is complete, and for each $v\notin\mathcal{U}$ either $v\in V(\mathrm{lk}(\mathcal{U}))$ or $v$ is not adjacent to any $u_i$. Thus, $v_{i_1,\dots,i_n}=v$ for each $v\in V(\mathrm{lk}(\mathcal{U}))$, with $0\leq i_k<o(u_k)$ for every $k=1,\dots,n$ and all other $v_{i_1,\dots,i_n}$ are pairwise distinct. We may therefore rewrite the presentation of $\ker(\varphi)$ as follows.
\begin{proposition}\label{Endomorphism}
	$\ker(\varphi)$ is a finitely presented group with presentation given by:
	$$\stretchleftright[1000]{<}{\begin{array}{l|cl}
			v_{i_1,\dots,i_n}, \text{ for }v\notin V(\mathrm{st}(\mathcal{U})) 	\text{ and } &v_{i_1,\dots,i_n}w_{i_1,\dots,i_n}=w_{i_1,\dots,i_n}v_{i_1,\dots,i_n}\text{ if }\lbrace v,w\rbrace\in E(\Gamma)\text{ and }v,w\notin \mathrm{st}(\mathcal{U})  \\
			0\leq i_k<o(u_k)\text{ with }k=1,\dots,n&vw_{i_1,\dots,i_n}=w_{i_1,\dots,i_n}v\text{ if }\lbrace v,w\rbrace\in E(\Gamma),v\in \mathrm{lk}(\mathcal{U}) \text{ and }w\notin \mathrm{st}(\mathcal{U})  \\
			&vw=wv\text{ if }\lbrace v,w\rbrace\in E(\Gamma)\text{ and }v,w\in \mathrm{lk}(\mathcal{U})  \\
			v \text{ for }v\in V(\mathrm{lk}(\mathcal{U}))& v_{i_1,\dots,i_n}^{o(v)}=1\text{ if }o(v)<\infty\text{ and }v\notin V(\mathrm{st}(\mathcal{U})) \\
			& v^{o(v)}=1\text{ if }o(v)<\infty\text{ and }v\in V(\mathrm{lk}(\mathcal{U}))
	\end{array}}{>}$$
In particular $\ker(\varphi)\cong \mathcal{G}_\Delta$, where $\Delta=\Gamma_{\mathcal{U}}\ast_{\lk(\mathcal{U})}\stackrel{m}{\cdots}\ast_{\lk(\mathcal{U})}\Gamma_{\mathcal{U}}$ and  $\Gamma_{\mathcal{U}}$ is the subgraph of $\Gamma$ spanned by
$V(\Gamma)\setminus\{u_1,\dots,u_n\}$.
\end{proposition}
\begin{proof} For each choice of indices $\{i_1,\dots,i_n\}$, the set
	$$
	\{v\mid v\in V(\lk(\mathcal{U}))\}
	\;\cup\;
	\{v_{i_1,\dots,i_n}\mid v\notin V(\st(\mathcal{U}))\}
	$$ induces a subgraph isomorphic to $\Gamma_{\mathcal{U}}$, so the claim follows.
\end{proof}
Hence, we need to understand the ordinary Betti numbers of amalgamated products of graphs. To this end, we will use the following proposition.

\begin{proposition}\label{AmalgamatedGraph}
	Let $\Gamma$ be a simplicial graph and let $\Delta\subset\Gamma$. Consider the amalgamated product $\Omega=\Gamma\ast_{\Delta}\stackrel{n}{\cdots}\ast_{\Delta}\Gamma$. Let $\Delta_1$ be a graph such that $\Delta\subset\Delta_1$ and $\widehat{\Delta}_1$ is contractible, and define $\gamma=\Gamma\ast_{\Delta}\Delta_1$. Then
	$$\overline{\beta}_i(\widehat{\Omega})=(n-1)\overline{\beta}_i(\widehat{\gamma})+\overline{\beta}_i(\widehat{\Gamma})\quad\text{for every } i\geq 0.$$
\end{proposition}

\begin{proof}
	We argue by induction on $i$. For $i=0$, the reduced zeroth Betti number equals the number of connected components minus $1$, so it suffices to count connected components. Let $r$ denote the number of connected components of $\widehat{\Gamma}$, and let $l$ denote the number of connected components of $\widehat{\Gamma}$ that intersect $\Delta$. Then the number of connected components of $\widehat{\Omega}$ is $l+(r-l)n$, while the number of connected components of $\widehat{\gamma}$ is $r-l+1$. Hence, the formula follows.
	
	Now assume $i\geq 1$. Applying the Mayer--Vietoris long exact sequence to $\widehat{\Omega}$ yields
	$$\cdots \rightarrow (n-1)\overline{\mathrm{H}}_i(\widehat{\Delta})\rightarrow n\overline{\mathrm{H}}_i(\widehat{\Gamma})\rightarrow \overline{\mathrm{H}}_i(\widehat{\Omega})\rightarrow \cdots.$$
	Similarly, applying the Mayer--Vietoris long exact sequence to $\widehat{\gamma}$ gives
	$$\cdots \rightarrow \overline{\mathrm{H}}_i(\widehat{\Delta})\rightarrow \overline{\mathrm{H}}_i(\widehat{\Gamma})\oplus \overline{\mathrm{H}}_i(\widehat{\Delta}_1)\rightarrow \overline{\mathrm{H}}_i(\widehat{\gamma})\rightarrow \cdots .$$
	Since $\widehat{\Delta}_1$ is contractible, multiplying this long exact sequence by $n-1$ we obtain
	$$\cdots \rightarrow (n-1)\overline{\mathrm{H}}_i(\widehat{\Delta})\rightarrow (n-1)\overline{\mathrm{H}}_i(\widehat{\Gamma})\rightarrow (n-1)\overline{\mathrm{H}}_i(\widehat{\gamma})\rightarrow \cdots .$$
	Adding a $\overline{\mathrm{H}}_i(\widehat{\Gamma})$ term to the middle group yields
	$$\cdots \rightarrow (n-1)\overline{\mathrm{H}}_i(\widehat{\Delta})\rightarrow n\overline{\mathrm{H}}_i(\widehat{\Gamma})\rightarrow (n-1)\overline{\mathrm{H}}_i(\widehat{\gamma})\oplus \overline{\mathrm{H}}_i(\widehat{\Gamma})\rightarrow \cdots .
	$$
	Thus, we obtain two long exact sequences in which, by the induction hypothesis, the groups in degrees $i$ and $i-1$ agree. By the Five Lemma, this induces an isomorphism in the middle groups
	$$\overline{\mathrm{H}}_i(\widehat{\Omega})\cong(n-1)\overline{\mathrm{H}}_i(\widehat{\gamma})\oplus\overline{\mathrm{H}}_i(\widehat{\Gamma}).$$
	Taking dimensions completes the proof.
\end{proof}
We can now compute the $\ell^2$-Betti numbers of every graph product of cyclic groups whose finite order vertices form a clique.
\begin{theorem}\label{BettiClique}
	Let $\mathcal{G}_\Gamma$ be a graph product of cyclic groups and assume that $\Gamma^{<\infty}$ is a clique. Then $\mathcal{G}_\Gamma$ is virtually a RAAG and:
	$$\beta_i^{(2)}(\mathcal{G}_\Gamma)=\frac{1}{\abs{\mathcal{G}_{\Gamma^{<\infty}}}}\left(\sum_{\Omega\subset\Gamma^{<\infty}}\left(\prod_{w\in V(\Gamma^{<\infty}\setminus\Omega)}(o(w)-1)\right)\beta_i^{(2)}(A_{\Gamma\setminus\Omega})\right)~\forall~i\geq 0$$
\end{theorem}
\begin{proof}
	Proceed by induction over $\abs{ V(\Gamma^{<\infty})}$. For the base case, say $\Gamma^{<\infty}=\lbrace v\rbrace$ and $o(v)=n$, let $\Delta:=\Gamma_v\ast_{\lk(v)}\stackrel{n}{\cdots}\ast_{\lk(v)}\Gamma_v$. By Proposition \ref{Endomorphism} applied to the map $\varphi:\mathcal{G}_\Gamma\to\mathbb{Z}/n\mathbb{Z}$ defined as $\varphi(v)=v$ and $\varphi(w)=0~\forall~w\neq v$, we have $\beta_i^{(2)}(\mathcal{G}_\Gamma)=\frac{1}{n}\beta_i^{(2)}(A_{\Delta})$. Using Theorem \ref{BettiRAAG} and Proposition \ref{AmalgamatedGraph}:
	$$\beta_i^{(2)}(A_\Delta)=(n-1)\beta_i^{(2)}(A_\Gamma)+\beta_i^{(2)}(A_{\Gamma_v})~\forall~i\geq 0$$
	and the base case follows. For the induction step suppose that the formula holds for all graph products $\mathcal{G}_\Lambda$ such that $\abs{ V(\Lambda^{<\infty})}=r$. Assume that $\mathcal{G}_\Gamma$ satisfies $\abs{ V(\Gamma^{<\infty})}=r+1$. Take $v\in V(\Gamma^{<\infty})$ and assume that $o(v)=n$ and define $\Delta$ and $\varphi$ as before. By proposition  \ref{Endomorphism} we have that $\beta_i^{(2)}(\mathcal{G}_\Gamma)=\frac{1}{n}\beta_i^{(2)}(\mathcal{G}_{\Delta})$ and $\Delta$ is a graph such that $\abs{ V(\Delta^{<\infty})}=r$. Apply the induction hypothesis to $\Delta$:
	$$\beta_i^{(2)}(\mathcal{G}_\Delta)=\frac{1}{\abs{\mathcal{G}_{\Delta^{<\infty}}}}\left(\sum_{\Omega\subset\Delta^{<\infty}}\left(\prod_{w\in V(\Delta^{<\infty}\setminus\Omega)}(o(w)-1)\right)\beta_i^{(2)}(A_{\Delta\setminus\Omega})\right)~\forall~i\geq 0$$
	Using Proposition \refeq{AmalgamatedGraph} and Theorem \ref{BettiRAAG}:
	$$\beta_i^{(2)}(A_{\Delta\setminus\Omega})=(o(v)-1)\beta_i^{(2)}(A_{\Gamma\setminus\Omega})+\beta_i^{(2)}(A_{\Gamma\setminus(\Omega\cup \lbrace v\rbrace)})~\forall~i\geq 0,\Omega\subset\Delta^{<\infty}$$
	and substituting this into the above equation the result follows.
\end{proof}
\subsection{$\lk(u_1)=\dots=\lk(u_n)$}\label{SectionLink}
We now assume that $\lk(u_1)=\dots=\lk(u_n):=L$. Hence, $\st(\mathcal{U})$ is the graph induced by $\lbrace u_1,\dots,u_n\rbrace\cup L$. Under this hypothesis $\lbrace u_i,u_j\rbrace\notin E(\Gamma)$ for any indices $i$ and $j$. Also, for each $v\notin\mathcal{U}$ either $v\in V(L)$ or $v$ is not adjacent to any $u_i$. Thus, $v_{i_1,\dots,i_n}=v$ for each $v\in V(\mathrm{lk}(\mathcal{U}))$, with $0\leq i_k<o(u_k)$ for every $k=1,\dots,n$. On the other hand, the fifth type of relators of Proposition \ref{Presentation} imply that:
$$\prescript{j}{}{X_{i_1,\dots,0,\dots,i_n}}=\left(\prod_{k=1}^{o(u_j)-1}
\prescript{j}{}{X_{i_1,\dots,k,\dots,i_n}}\right)^{-1}$$
Hence, we can delete the generator $\prescript{j}{}{X_{i_1,\dots,i_n}}$ with a zero on the index $j$ for each $1\leq j\leq o(u_j)-1$. Moreover, the relator $v_{i_1,\dots,i_n}
=v_{i_1,\dots,i_j+1,\dots,i_n}^{\prescript{j}{}{X_{i_1,\dots,0,\dots,i_n}}}$ is easily deduced from the relators of the form $v_{i_1,\dots,i_n}
=v_{i_1,\dots,i_j+1,\dots,i_n}^{\prescript{j}{}{X_{i_1,\dots,k,\dots,i_n}}}$ with $k=1,\dots,o(u_j)-1$, so we can delete it.

\begin{lemma}
	There are
	$$
	A:=m\left(n-1-\sum_{j=1}^{n}\frac{1}{o(u_j)}\right)+1
	$$
	non-zero generators of the form $\prescript{j}{}{X_{i_1,\dots,i_n}}$. They are exactly the ones satisfying that $1\leq j\leq n-1$, $0\leq i_k<o(u_k)$ for each $1\leq k\leq n$, $i_j\neq0$ and $(i_{j+1},\dots,i_n)\neq(0,\dots,0)$.
\end{lemma}

\begin{proof}
	Two generators of a graph product of cyclic groups commute if and only if they are connected in the defining graph. Hence, since $u_i$ is not connected to $u_j$ for any indices $i$ and $j$, we have $[u_{j+1}^{i_{j+1}}\cdots u_n^{i_n},u_j]=1$
	if and only if $(i_{j+1},\dots,i_n)=(0,\dots,0)$. Thus, for each $j$, the number of non-zero distinct generators $\prescript{j}{}{X_{i_1,\dots,i_n}}$ equals the number of $n$-tuples
	$(i_1,\dots,i_n)$ satisfying
	$$
	0\leq i_k<o(u_k)\ \text{ for all } 1\leq k\leq n,\quad
	i_j\neq 0,\quad
	(i_{j+1},\dots,i_n)\neq (0,\dots,0).
	$$
	The number of such $n$-tuples is
	$$o(u_1)\cdots o(u_{j-1})\bigl(o(u_j)-1\bigr)\bigl(o(u_{j+1})\cdots o(u_n)-1\bigr).	$$
	Hence, the total number of non-zero generators is
	$$\sum_{j=1}^{n-1}o(u_1)\cdots o(u_{j-1})\bigl(o(u_j)-1\bigr)\bigl(o(u_{j+1})\cdots o(u_n)-1\bigr).$$
	A straightforward calculation shows that this sum equals $A$.
\end{proof}
Hence, the presentation for $\ker(\varphi)$ can be expressed as follows.
\begin{proposition}\label{EndomorphismLink}
		$\ker(\varphi)$ is a finitely presented group with presentation given by:
	$$\stretchleftright[1000]{<}{\begin{array}{l|cl}
			v_{i_1,\dots,i_n}, \text{ for }v\notin V(\st(\mathcal{U})) 	\text{ and } &v_{i_1,\dots,i_n}w_{i_1,\dots,i_n}=w_{i_1,\dots,i_n}v_{i_1,\dots,i_n}\text{ if }\lbrace v,w\rbrace\in E(\Gamma)\text{ and }v,w\notin V(\mathrm{st}(\mathcal{U}))\\
			0\leq i_k<o(u_k)\text{ with }k=1,\dots,n&vw_{i_1,\dots,i_n}=w_{i_1,\dots,i_n}v\text{ if }\lbrace v,w\rbrace\in E(\Gamma),v\in V(L) \text{ and }w\notin V(\mathrm{st}(\mathcal{U})) \\
			v \text{ for }v\in V(L)& vw=wv\text{ if }\lbrace v,w\rbrace\in E(\Gamma)\text{ and }v,w\in V(L)\\
			\prescript{j}{}{X_{i_1,\dots,i_n}}\text{ for every }1\leq j\leq n-1,&v\prescript{j}{}{X_{i_1,\dots,i_n}}=	\prescript{j}{}{X_{i_1,\dots,i_n}}v\text{ if }v\in V(L)\\
			0\leq i_k<o(u_k)\text{ with }1\leq k\leq n,i_j\neq0&v_{i_1,\dots,i_n}^{o(v)}=1\text{ if }o(v)<\infty\text{ and }v\notin V(\mathrm{st}(\mathcal{U}))\\
			\text{ and }(i_{j+1},\dots,i_n)\neq(0,\dots,0)&v^{o(v)}=1\text{ if }o(v)<\infty\text{ and }v\in V(L)\\
	\end{array}}{>}$$
	Hence $\ker(\varphi)$ is a graph product of cyclic groups
\end{proposition}
Assume that $\ker(\varphi)=\mathcal{G}_\Omega$ and let us compute the reduced Betti numbers of $\widehat{\Omega}$. Let $$\Delta=L\cup\lbrace\prescript{j}{}{X_{i_1,\dots,k,\dots,i_n}}\mid 1\leq j\leq n-1,i_j\neq0\rbrace$$ 
and let $\Delta_1^{i_1,\dots,i_n}$ be the graph induced by $\Delta\cup\lbrace v_{i_1,\dots, i_n}\mid v\notin V(\st(\mathcal{U}))\rbrace$ for each $0\leq i_j<o(u_j)$ and $1\leq j\leq n$. All copies of  $\Delta_1^{i_1,\dots,i_n}$ are isomorphic to a graph, let us name it $\Delta_1$. In this context one sees that $\Omega=\Delta_1\ast_\Delta\stackrel{m}{\cdots}\ast_\Delta\Delta_1$. Define $\Delta_2$ to be the graph obtained after adding a new vertex $r$ to $\Delta_1$ and connecting $r$ to all $v\in V(\Delta)$. Then, $\Delta_2=\Delta_1\ast_{\lk(r)}\st(r)=\Delta_1\ast_{\Delta}\st(r)$ and $\widehat{\st(r)}$ is contractible. Hence, by Proposition \ref{AmalgamatedGraph} it follows that:
	$$\overline{\beta}_i(\widehat{\Omega})=(m-1)\overline{\beta}_i(\widehat{\Delta}_2)+\overline{\beta}_i(\widehat{\Delta}_1)\quad\text{for every } i\geq 0.$$
Let $\Gamma_{\mathcal{U}_1}$ be the subgraph induced by $\Gamma\setminus\lbrace u_2,\dots, u_n\rbrace$. Each $\prescript{j}{}{X_{i_1,\dots,k,\dots,i_n}}$ induces a copy of $\Gamma_{\mathcal{U}_1}$ in $\Delta_1$. Hence $\Delta_1=\Gamma_{\mathcal{U}_1}\ast_{\Gamma_{\mathcal{U}}}\stackrel{A}{\cdots}\ast_{\Gamma_{\mathcal{U}}}\Gamma_{\mathcal{U}_1}$. Let $\Gamma_1$ be the graph obtained from $\Gamma_{\mathcal{U}_1}$ after adding a new vertex $l$ to $\Gamma_{\mathcal{U}_1}$ and connecting it to all $v\in V(\Gamma_{\mathcal{U}})$. Then, $\Gamma_1=\Gamma_{\mathcal{U}_1}\ast_{\lk(l)}\st(l)=\Gamma_{\mathcal{U}_1}\ast_{\Gamma_\mathcal{U}}\st(l)$ and $\widehat{\st(r)}$ is contractible. Hence, by Proposition \ref{AmalgamatedGraph} it follows that:
$$\overline{\beta}_i(\widehat{\Delta_1})=(A-1)\overline{\beta}_i(\widehat{\Gamma}_1)+\overline{\beta}_i(\widehat{\Gamma}_{\mathcal{U}_1})\quad\text{for every } i\geq 0.$$
Let $\Gamma_{\mathcal{U}_2}$ be the graph induced by $\Gamma\setminus\lbrace u_3,\dots, u_n\rbrace$ after joining $u_1$ and $u_2$. Each $\prescript{j}{}{X_{i_1,\dots,k,\dots,i_n}}$ induces a copy of $\Gamma_{\mathcal{U}_2}$ in $\Delta_2$. Hence $\Delta_2=\Gamma_{\mathcal{U}_2}\ast_{\Gamma_{\mathcal{U}_1}}\stackrel{A}{\cdots}\ast_{\Gamma_{\mathcal{U}_1}}\Gamma_{\mathcal{U}_2}$. Let $\Gamma_2$ be the graph obtained from $\Gamma_{\mathcal{U}_2}$ after adding a new vertex $w$ to $\Gamma_{\mathcal{U}_2}$ and connecting it to all $v\in V(\Gamma_{\mathcal{U}_1})$. Then, $\Gamma_2=\Gamma_{\mathcal{U}_2}\ast_{\lk(w)}\st(w)=\Gamma_{\mathcal{U}_2}\ast_{\Gamma_{\mathcal{U}_1}}\st(w)$ and $\widehat{\st(w)}$ is contractible. Hence, by Proposition \ref{AmalgamatedGraph} it follows that:
$$\overline{\beta}_i(\widehat{\Delta_2})=(A-1)\overline{\beta}_i(\widehat{\Gamma}_2)+\overline{\beta}_i(\widehat{\Gamma}_{\mathcal{U}_2})\quad\text{for every } i\geq 0.$$
Combining all formulas, we get an expression for the reduced Betti numbers of $\widehat{\Omega}$. However, this expression can be simplified using the next observation.
\begin{lemma}
	\begin{enumerate}
		\item $\widehat{\Gamma_2}$ is contractible.
		\item $\overline{\beta}_i(\widehat{\Gamma}_1)=\overline{\beta}_{i-1}(\widehat{L})$ for all $i\geq 0$.
		\item $\overline{\beta}_i(\widehat{\Gamma}_{\mathcal{U}_2})=\overline{\beta}_i(\widehat{\Gamma}_{\mathcal{U}_1})$ for all $i\geq 0$.
	\end{enumerate}
\end{lemma}
\begin{proof}
	For $\widehat{\Gamma}_2$ we can decompose it as $\widehat{\Omega}_1\cup\widehat{\Omega}_2$ where $\Omega_1$ is the graph induced by $\Gamma_2\setminus\lbrace u_2\rbrace$ and $\widehat{\Omega_2}$ is the graph induced by $L\cup \lbrace u_1,u_2\rbrace$.  Hence, we obtain the following pushout diagram:
	$$ \begin{tikzcd}
		\widehat{\Omega}_1\cap\widehat{\Omega}_2\arrow[hook]{r} \arrow[swap,hook]{d} & 	\widehat{\Omega}_1\simeq\bullet \arrow[hook]{d}\\%
			\widehat{\Omega}_2\simeq\bullet \arrow[hook]{r}& \widehat{\Gamma}_2
	\end{tikzcd}$$
	which implies that $\widehat{\Gamma}_2$ is homotopy equivalent to the suspension $\mathbb{S}^0\ast\widehat{\Omega}_1\cap\widehat{\Omega}_2$. Since $\widehat{\Omega}_1\cap\widehat{\Omega}_2$ equals $\widehat{\Omega}_3$, where $\Omega_3$ is the graph induced by $L\cup\lbrace u_1\rbrace=\st(u_1)$. Hence $\widehat{\Omega}_3$ is contractible and $1.$ follows.
	
	In the same way, we can decompose $\widehat{\Gamma}_1$ as $\widehat{\Omega}_1\cup\widehat{\Omega}_2$, where $\Omega_1$ is the subgraph induced by $L\cup\lbrace u_1\rbrace$ and $\Omega_2$ is the subgraph induced by $\Gamma_\mathcal{U}\cup\lbrace l\rbrace$. Since both $\widehat{\Omega}_1$ and $\widehat{\Omega}_2$ are contractible it follows that $\widehat{\Gamma}_2$ is homotopy equivalent to the suspension $\mathbb{S}^0\ast\widehat{\Omega}_1\cap\widehat{\Omega}_2$. Since $\widehat{\Omega}_1\cap\widehat{\Omega}_2$ equals $\widehat{L}$ we have $\overline{\mathrm{H}}_i(\widehat{\Gamma}_1)=\overline{\mathrm{H}}_i(\mathbb{S}^0\ast\widehat{L})=\overline{\mathrm{H}}_{i-1}(\widehat{L})$ and $2.$ follows.
	
	For $3.$ observe that in $\widehat{\Gamma}_{\mathcal{U}_2}$ $u_2$ is dominated by $u_1$, so we can collapse $u_2$ into $u_1$ and $\widehat{\Gamma}_{\mathcal{U}_2}\simeq \widehat{\Gamma}_{\mathcal{U}_1}$.
	\end{proof}
Applying the lemma we can compute a formula for the $\ell^2$-Betti numbers  when $V(\Gamma^{<\infty})=\mathcal{U}$.
\begin{theorem}\label{BettiLink}
	Assume that $V(\Gamma^{<\infty})=\lbrace u_1,\dots,u_n\rbrace$ and that  $\lk(u_1)=\dots=\lk(u_n):=L$. Then:
	$$\beta_i^{(2)}(\mathcal{G}_\Gamma)=\beta_i^{(2)}(A_{\Gamma_{\mathcal{U}_1}})+\left(n-1-\sum_{j=1}^{n}\frac{1}{o(u_j)}\right)\beta_{i-1}^{(2)}(A_{L})~\forall~i\geq 0$$
	where $\Gamma_{\mathcal{U}_1}$ is the graph obtained after deleting $u_2,\dots, u_n$ from $\Gamma$.
\end{theorem}
In particular, the first $\ell^2$-Betti number of these groups coincides with that of the associated right-angled Artin group.

\begin{corollary}
	Assume that $V(\Gamma^{<\infty})={u_1,\dots,u_n}$ and that
	$\lk(u_1)=\cdots=\lk(u_n)=:L$.
	Then
	$$\beta_1^{(2)}(\mathcal{G}_\Gamma)=\beta_1^{(2)}(A_\Gamma).$$
\end{corollary}

\begin{proof}
	By Theorem \ref{BettiRAAG}, the first $\ell^2$-Betti number of a right-angled Artin group is determined by the number of connected components of its defining graph. Since $\Gamma_{\mathcal U_1}$ and $\Gamma$ have the same number of connected components, we have $\beta_1^{(2)}(A_{\Gamma_{\mathcal U_1}})=\beta_1^{(2)}(A_\Gamma)$.
	Moreover, $\beta_0^{(2)}(A_L)=0$ since $A_L$ is infinite. Thus, by the preceding result,
	$\beta_1^{(2)}(\mathcal{G}_\Gamma)=\beta_1^{(2)}(A_{\Gamma_{\mathcal U_1}})=\beta_1^{(2)}(A_\Gamma)$, as desired.
\end{proof}

\subsection{Example}
Let us see an example of a graph product of cyclic groups for which Theorems \ref{BettiClique} and \ref{BettiLink} do not apply but where we can compute the $\ell^2$-Betti numbers by applying the ideas of Sections \ref{SectionClique} and \ref{SectionLink}.
\begin{example}
	Consider the graph product of cyclic groups $\mathcal{G}_\Gamma$ with graph:
		$$\begin{tikzpicture}[main/.style = {draw, circle},node distance={15mm}] 
		\node[label=left:{$c$}][main] (1)  {}; 
		\node[label={$e$}][main] (2) [right of=1] {};
		\node[label=below:{$d$}][main] (3) [below of=1] {}; 
		\node[label={$b$}][main] (4) [above of=1] {};
		\node[label={$f$}][main] (5) [right of=2] {};
		\node[label={$g$}][main] (6) [above of=5] {};
		\node[label={$a$}][main] (7) [left of=1] {};
		\draw[-] (1) -- (2);
		\draw[-] (1) -- (3);
		\draw[-] (1) -- (4);
		\draw[-] (2) -- (5);
		\draw[-] (3) -- (5);
		\draw[-] (3) -- (7);
		\draw[-] (4) -- (5);
		\draw[-] (4) -- (7);
	\end{tikzpicture} $$
	where $o(b)=o(c)=o(d)=o(e)=2$ and $o(a)=o(f)=o(g)=\infty$. Consider the map $\varphi:\mathcal{G}_\Gamma\to\left(\mathbb{Z}/2\mathbb{Z}\right)^2$ given by $\varphi(b)=(1,0),\varphi(d)=(0,1)$ and $\varphi(a)=\varphi(c)=\varphi(e)=\varphi(f)=\varphi(g)=0$. By Proposition \ref{EndomorphismLink} $\ker(\varphi)\cong\mathcal{G}_{\Gamma_1}$, where $\Gamma_1$ is the next graph:
		$$\begin{tikzpicture}[main/.style = {draw, circle},node distance={15mm}] 
		\node[label=below:{$c$}][main] (1)  {}; 
		\node[label=right:{$e_{1,0}$}][main] (2) [above  of=1] {};
		\node[label=right:{$e_{0,1}$}][main] (3) [right  of=2] {};
		\node[label=right:{$e_{1,1}$}][main] (4) [right  of=3] {};
		\node[label=right:{$e_{0,0}$}][main] (5) [left   of=2] {};
		\node[label={$f$}][main] (6) [above of=2] {};
		\node[label={$g_{0,0}$}][main] (7) [right of=6] {};
		\node[label={$g_{1,0}$}][main] (8) [right of=7] {};
		\node[label={$g_{0,1}$}][main] (9) [right of=8] {};
		\node[label={$g_{0,0}$}][main] (10) [right of=9] {};
		\node[label={$\prescript{b}{}{X_{1,1}}$}][main] (11) [left of =5] {};
		\node[label={$a$}][main] (12) [left of=11] {};				
		\draw[-] (1) -- (2);
		\draw[-] (1) -- (3);
		\draw[-] (1) -- (4);
		\draw[-] (1) -- (5);
		\draw[-] (1) -- (11);
		\draw[-] (6) -- (2);
		\draw[-] (6) -- (3);
		\draw[-] (6) -- (4);	
		\draw[-] (6) -- (5);	
		\draw[-] (6) -- (11);
		\draw[-] (11) -- (12);	
	\end{tikzpicture} $$
	where $o(c)=o(e_{i,j})=2$ and the other vertices have infinite order. Now consider the map $\varphi_1:\mathcal{G}_{\Gamma_1}\to\left(\mathbb{Z}/2\mathbb{Z}\right)^4$ given by sending each $e_{i,j}$ to one of the generators of $\left(\mathbb{Z}/2\mathbb{Z}\right)^4$ and sending all the other generators to $0$. By Proposition \ref{EndomorphismLink} $\ker(\varphi_1)\cong\mathcal{G}_{\Gamma_2}$. The graph $\Gamma_2$ is defined as follows:
	\begin{itemize}
		\item The vertex set consists of $c,f$, $16$ vertices of the form $(\prescript{b}{}{X_{1,1}})_{(i_1,i_2,i_3,i_4)}$,  $16$ vertices of the form $a_{(i_1,i_2,i_3,i_4)}$, $64$ vertices of the form $\left(g_{j,k}\right)_{(i_1,i_2,i_3,i_4)}$, $17$ vertices of the form $\prescript{e_{i,j}}{}{X}_{(i_1,i_2,i_3,i_4)}$.
		\item All vertices of the form $(\prescript{b}{}{X_{1,1}})_{(i_1,i_2,i_3,i_4)}$ and $\prescript{e_{i,j}}{}{X}_{(i_1,i_2,i_3,i_4)}$ are connected to both $c$ and $f$. Each $a_{(i_1,i_2,i_3,i_4)}$ is connected to $(\prescript{b}{}{X_{1,1}})_{(i_1,i_2,i_3,i_4)}$.There are no more edges in the graph.
	\end{itemize}
	The only vertex in $\Gamma_2$ with finite order is $c$, hence we can apply Theorem \ref{BettiClique} and the $\ell^2$-Betti numbers of $\mathcal{G}_\Gamma$ are:
	$$\beta_i^{(2)}(\mathcal{G}_\Gamma)=\frac{1}{64}\beta_i^{(2)}(\mathcal{G}_{\Gamma_2})=\frac{1}{128}\left(\beta_i^{(2)}(A_{\Gamma_2})+\beta_i^{(2)}(A_{(\Gamma_2)_c})\right)$$
	and it is easy to check that:
	$$\beta_i^{(2)}(A_{\Gamma_2})=\left\{\begin{array}{ll}
		64 & \text{if } i=1 \\
		32 & \text{if } i=2 \\
		0 & \text{otherwise} \\
	\end{array}\right.
	~~~~~~\beta_i^{(2)}(A_{(\Gamma_2)_c})=\left\{\begin{array}{ll}
		64 & \text{if } i=1 \\
		0 & \text{otherwise} \\
	\end{array}\right. $$
	Hence:
		$$\beta_i^{(2)}(\mathcal{G}_{\Gamma})=\begin{dcases}
		1 & \text{if } i=1 \\
		\frac{1}{4} & \text{if } i=2 \\
		0 & \text{otherwise} \\
	\end{dcases}$$
\end{example} 
\section{The first $\ell^2$-Betti number and derivations}
Let $G$ be a group and let $M$ be a $G$-module. The first cohomology group of $G$ with coefficients in $M$ is defined by

$$\mathrm{H}^1(G,M)=\frac{\mathrm{Der}(G,M)}{\mathrm{InnDer}(G,M)},$$

where
\begin{itemize}
	\item $\mathrm{Der}(G,M)$ denotes the set of derivations from $G$ to $M$, that is, maps $f:G\to M$ satisfying $f(gh)=f(g)+g f(h)$ for all $g,h\in G$;
	\item $\mathrm{InnDer}(G,M)$ denotes the subgroup of inner derivations, consisting of maps $f_m:G\to M$ of the form $f_m(g)=gm-m$ for some $m\in M$.
	
\end{itemize}

Let $G$ be a group and let $\mathcal{N}(G)$ be the von Neumann algebra, i.e. the algebra of bounded $G$-equivariant operators on $L^2(G)$. The Linnell ring $\mathcal{D}(G)$ is defined to be the division closure of $\mathbb{C}G$ in $\mathrm{Ore}(\mathcal{N}(G))$. For more details on the definition see \cite{Luck}.
\begin{definition}
	Let $G$ be a torsion-free group. We say that $G$ satisfies the \textbf{Atiyah conjecture} if the Linnell ring $\mathcal{D}(G)$ is a division ring.
\end{definition}
\begin{remark}
	Although this is not the standard formulation of the Atiyah conjecture, a result of Linnell (see \cite{Linell}) shows that, for torsion-free groups, it is equivalent to the Atiyah conjecture over $\mathbb{Q}$.
\end{remark}
Now, assume that $G$ satisfies the strong Atiyah conjecture. By \cite[Lemma 10.28 (3)]{Luck} the first $\ell^2$-Betti number of $G$ can be computed as:

$$\beta_1^{(2)}(G)=\dim_{\mathcal{D}(G)} \mathrm{H}^1(G,\mathcal{D}(G))=\dim_{\mathcal{D}(G)} \mathrm{Der}(G,\mathcal{D}(G))-\dim_{\mathcal{D}(G)} \mathrm{InnDer}(G,\mathcal{D}(G)).$$
It is straightforward to see that $\mathrm{InnDer}(G,\mathcal{D}(G))\cong\mathcal{D}(G)$, so:
$$\beta_1^{(2)}(G)=\dim_{\mathcal{D}(G)} \mathrm{Der}(G,\mathcal{D}(G))-1.$$
\begin{theorem}
	The following families of groups satisfy the Atiyah conjecture:
	\begin{enumerate}
		\item Locally indicable groups. \cite[Jaikin-Zapirain-López-Álvarez]{Jaikin-Lopez 2020} 
		\item RFRS groups. \cite[Kielak]{Kielak}
	\end{enumerate}
\end{theorem}
Recall that a group $G$ is \textbf{locally indicable} if every finitely generated subgroup is either trivial or admits an epimorphism onto $\mathbb{Z}$. As a first application of this method we can compute the $\ell^2$-Betti numbers of all torsion-free one-relator groups.
\begin{lemma}
	Let $G=\langle x_1,\ldots,x_d\mid r=1\rangle$ be a torsion-free one-relator group. Then
	$$
	\beta_i^{(2)}(G)=
	\begin{cases}
		d-2, & \text{if } i=1,\\
		0, & \text{otherwise}.
	\end{cases}
	$$
	
\end{lemma}

\begin{proof}
	By a theorem of Brodskii and Howie, every torsion-free one-relator group is locally indicable (see \cite{Brodski} and \cite{Howie}). In particular, $G$ satisfies the Atiyah conjecture, and hence the Linnell division ring $\mathcal{D}(G)$ is defined.
	
	Let	$f\in\operatorname{Der}(G,\mathcal{D}(G))$ and write the relator as a reduced word $r=a_1\cdots a_n$, where $a_i\in\{x_1^{\pm1},\ldots,x_d^{\pm1}\}$.
	
	Since $r=1$ in $G$, we have $0=f(r)
	=f(a_1)+a_1f(a_2)+\cdots+a_1\cdots a_{n-1}f(a_n)$. Moreover, for every generator $x_i$ we have $f(x_i^{-1})=-x_i^{-1}f(x_i)$. Choose a generator $x_j$ which occurs in the relator $r$. The above relation therefore determines $f(x_j)$ uniquely in terms of the values	$f(x_i)$ for $i\neq j$.	Consequently, $\operatorname{Der}(G,\mathcal{D}(G))\cong \mathcal{D}(G)^{d-1}$ and $\beta_1^{(2)}(G)=d-2$.
	
	Finally, torsion-free one-relator groups have cohomological dimension at most $2$ with the presentation complex as the classifying space (cf. \cite{Putnam} Theorem 7.1). Hence, 
	$\beta_i^{(2)}(G)=0$ for all $i>2$ and $\chi(G)=2-d$. Since $G$ is infinite, $\beta_0^{(2)}(G)=0$, and therefore $\beta_2^{(2)}(G)=\beta_1^{(2)}(G)+\chi(G)=0$.
\end{proof}

\subsection{Twisted RAAGs}\label{Twisted RAAGs}
A mixed simplicial graph $\Gamma$ is a simplicial graph in which the edge set decomposes as $E(\Gamma)=U(\Gamma)\cup D(\Gamma)$. The edges in $U(\Gamma)$ are undirected and the edges in $D(\Gamma)$ are directed. The \textbf{twisted RAAG} associated to $\Gamma$ is defined as the group with the following presentation:
$$T_\Gamma=\langle v\in V(\Gamma)\mid uv=vu~\mathrm{ if }~e=\lbrace u,v\rbrace\in U(\Gamma), uvu=v~\mathrm{ if }~e=\begin{tikzpicture}
	\node[circle, fill, inner sep=2pt, label=above:$u$] (u) at (0,0) {};
	\node[circle, fill, inner sep=2pt, label=above:$v$] (v) at (1,0) {};
	\draw[->] (u) -- (v);
\end{tikzpicture}\in D(\Gamma)\rangle$$
In the case where $D(\Gamma)=\emptyset$ the twisted RAAG is just simply a RAAG. An \textbf{oriented cycle} in a mixed simplicial graph is a cycle in which all edges are oriented consistently around the cycle.
\begin{theorem}[Antolín-Blufstein-Paris, \cite{Antolin 2025}] A twisted RAAG is locally indicable if and only if it contains no oriented cycles.
\end{theorem}
Hence, we may assume that $T_\Gamma$ is locally indicable, so $\mathcal{D}(T_\Gamma)$ exists and we can study the module $\mathrm{Der}(T_\Gamma,\mathcal{D}(T_\Gamma))$. To do it we consider a previous lemma.

\begin{lemma}\label{Derivation Lemma}
	Let $G$ be a torsion-free group satisfying the Atiyah conjecture and $f\in  \mathrm{Der}(T_\Gamma,\mathcal{D}(T_\Gamma))$.
	\begin{enumerate}
		\item If $u,v\in G$ are such that $uv=vu$ then $(1-v)^{-1}f(v)=(1-u)^{-1}f(u)$.
		\item f $u,v\in G$ are such that $uvu=v$ then $(1-v)^{-1}f(v)=(1-u)^{-1}f(u)$.
	\end{enumerate} 
\end{lemma}
\begin{proof}
	If $uv=vu$ then $f(uv)=f(vu)\Rightarrow f(u)+uf(v)=f(v)+vf(u)\Rightarrow (1-u)f(v)=(1-v)f(u)$. But $1-x$ is invertible in $\mathcal{D}(G)$ for all $x\neq 1$, so $(1-v)^{-1}f(v)=(1-u)^{-1}f(u)$.
	
	If $uvu=v$ then $f(uvu)=f(v)\Rightarrow f(u)+uf(v)+uvf(u)=f(v)\Rightarrow (1+uv)f(u)=(1-u)f(v)$. Observe that $(1+uv)(1-u)=(1-u)(1-v)$ and so $(1-u)(1-v)(1-u)^{-1}f(u)=(1-u)f(v)\Rightarrow (1-v)^{-1}f(v)=(1-u)^{-1}f(u)$.
\end{proof}

Now, let $\Gamma'\subset\Gamma$ be a connected component of $\Gamma$. Then $(1-v)^{-1}f(v)$ is a constant value for all $v\in V(\Gamma')$. Fix a vertex $u\in V(\Gamma')$ and fix a value $f(u)\in\mathcal{D}(T_\Gamma)$, then define $c:=(1-u)^{-1}f(u)$ and $f(v)=(1-u)c$ for all $v\in V(\Gamma')$. Hence, the value of $f$ on each connected component depends solely on the value on one vertex of the component, so $\mathrm{Der}(T_\Gamma,\mathcal{D}(T_\Gamma))\cong\mathcal{D}(T_\Gamma)^{\beta_0(\Gamma)}$. Hence:
\begin{proposition}
	If $T_\Gamma$ is a twisted RAAG with no oriented cycles then $\beta_1^{(2)}(T_\Gamma)=\overline{\beta}_0(\Gamma)$
\end{proposition}
This gives an alternative proof for the case $n=1$ of Theorem \ref{BettiRAAG}.

\begin{remark}
Let $A_\Gamma$ be an Artin group that is torsion-free and satisfies the Atiyah conjecture. Let $f\in \operatorname{Der}(A_\Gamma,\mathcal{D}(A_\Gamma))$. Every canonical relator has the form $$\langle u,v\rangle_n=\langle v,u\rangle_n$$ for some $n$, where $\langle u,v\rangle_n=\underbrace{uvuv\cdots}_{n\text{ letters}}$. If $n$ is even, applying $f$ to the relator and using the derivation rule gives
	$$\bigl(1-v+uv-vuv+\cdots-\langle v,u\rangle_{n-1}\bigr)f(u)=\bigl(1-u+vu-uvu+\cdots-\langle u,v\rangle_{n-1}\bigr)f(v).$$
	
	Moreover, $(1-v+uv-vuv+\cdots-\langle v,u\rangle_{n-1})(1-u)=(1-u+vu-uvu+\cdots-\langle u,v\rangle_{n-1})(1-v)$. These two identities imply $(1-v)^{-1}f(v)=(1-u)^{-1}f(u)$. In the case where $n$ is odd we can similarly prove that $(1-v)^{-1}f(v)=(1-u)^{-1}f(u)$.
	
	Hence, if $A_\Gamma$ is torsion-free and satisfies the Atiyah conjecture, a similar argument as in the twisteed RAAG case shows that $\beta_1^{(2)}(A_\Gamma)=\overline{\beta}_0(\Gamma)$. This is also a particular case of Davis-Leary theorem for $n=1$ (see Theorem \ref{BettiRAAG}). Some families of torsion-free Artin groups satisfying the Atiyah conjecture are RAAGs, even Artin groups of FC-type \cite{Blasco-Martinez} and even Artin groups of large type \cite{Blasco}.
\end{remark}
\subsection{Graph product of cyclic groups with $\abs{V(\Gamma^{<\infty})}=2$}
Assume $V(\Gamma^{<\infty})=\lbrace u_1,u_2\rbrace$ and that $d(u_1,u_2)\geq2$, since the $\ell^2$-Betti numbers of the case with $u_1$ and $u_2$ connected is known from Theorem \ref{BettiClique}. Moreover, we may assume $u_1$ and $u_2$ lie in the same connected component of $\Gamma$, otherwise we could write $\mathcal{G}_\Gamma=\mathcal{G}_{\Gamma_1}\ast\mathcal{G}_{\Gamma_2}$ with $u_i\in V(\Gamma_i)$ and use Theorem \ref{BettiClique} and the formula for the $\ell^2$-Betti numbers of free products (cf. \cite{Kammeyer} Theorem 4.15.ii) to compute $\beta_i^{(2)}(\mathcal{G}_\Gamma)$.

Define $\varphi:\mathcal{G}_\Gamma\to\mathbb{Z}/o(u_1)\mathbb{Z}\times \mathbb{Z}/o(u_2)\mathbb{Z}$ by $\varphi(u_1)=(1,0)$ and $\varphi(u_2)=(0,1)$. By Proposition \ref{Presentation} $\ker(\varphi)$ has the following presentation:
$$\stretchleftright[1000]{<}{\begin{array}{l|cl}
		v_{i_1,i_2}, \text{ for }v\neq u_1,u_2	\text{ and } &v_{i_1,i_2}w_{i_1,i_2}=w_{i_1,i_2}v_{i_1,i_2}\text{ if }\lbrace v,w\rbrace\in E(\Gamma) \\
		0\leq i_k<o(u_k)\text{ with }k=1,2&v_{i_1,i_2}
		=v_{i_1+1,i_2}^{X_{i_1,i_2}}\text{ if }\lbrace v,u_1\rbrace\in E(\Gamma) \\
		X_{i_1,i_2} & v_{i_1,i_2}=v_{i_1,i_2+1}\text{ if }\lbrace v,u_2\rbrace\in E(\Gamma)\\
			0<i_k<o(u_k)\text{ with }k=1,2& v_{i_1,0}=v_{i_1+1,0}\text{ if }\lbrace v,u_1\rbrace\in E(\Gamma)\\
\end{array}}{>}
$$

Let $A_\Delta$ be the RAAG generated by all $v_{i_1,i_2}$ for $v\in V(\Gamma)\setminus\lbrace u_1,u_2\rbrace$. Then, $\ker(\varphi)$ can be constructed as iterated HNN-extension starting from $A_\Delta$. Defining $K=A_{\mathrm{lk}(u_1)}$ we have:
$$\ker(\varphi)=(\cdots((A_\Delta\ast_K)\ast_K)\cdots\ast_K)\ast_K$$
where in each step we are adding a $X_{i_1,i_2}$ generator. By \cite{Huang} graph products of cyclic group are virtually RFRS, so they satisfy the Atiyah conjecture and thus $\ker(\varphi)$ being a subgroup also satisfies the Atiyah conjecture by \cite[Lemma 10.4]{Luck}. Moreover, $\ker(\varphi)$ is torsion-free being the iterated HNN extension of a torsion-free group $A_\Delta$. Thus, $\mathcal{D}(\ker(\varphi))$ is a division ring.

To compute $\dim_{\mathcal{D}(\ker(\varphi))}
\mathrm{Der}\bigl(\ker(\varphi),\mathcal{D}(\ker(\varphi))\bigr)$, we first study the connected components of $\Delta$ and then analyse the effect on the derivations of adding relators of the form $v_{i_1,i_2}=v_{i_1+1,i_2}^{X_{i_1,i_2}}$. Observe that there are four types of connected components in $\Gamma_{u_1u_2}$:

\begin{itemize}
	\item \textbf{Type 1:} Connected components that are components of both $\Gamma$ and $\Gamma_{u_1u_2}$. There are $\beta_1^{(2)}(A_\Gamma)$ such components.
	\item \textbf{Type 2:} Connected components of $\Gamma_{u_1}$ that are not connected components of $\Gamma$. There are
	$\beta_1^{(2)}(A_{\Gamma_{u_1}})-\beta_1^{(2)}(A_\Gamma)$ such components.
	\item \textbf{Type 3:} Connected components of $\Gamma_{u_2}$ that are not connected components of $\Gamma$. There are $\beta_1^{(2)}(A_{\Gamma_{u_2}})-\beta_1^{(2)}(A_\Gamma)$ such components.
	\item \textbf{Type 4:} Connected components of $\Gamma_{u_1u_2}$ that are not connected components of either $\Gamma_{u_1}$ or $\Gamma_{u_2}$.	
\end{itemize}

For each $v_{i_1,i_2}$ appearing in a relator of the form $v_{i_1,i_2}=v_{i_1+1,i_2}^{X_{i_1,i_2}}$, we have $\{v,u_1\}\in E(\Gamma)$.
Consequently, $v$ belongs to a connected component of Type 2 or Type 4. It follows that the elements $X_{i_1,i_2}$ do not affect the connected components of Type 1 or Type 3. Therefore, each connected component of $\Delta$ arising from a component of Type 1 or Type 3 contributes one dimension to $\mathrm{Der}\bigl(\ker(\varphi),\mathcal{D}(\ker(\varphi))\bigr)$. We thus begin by considering these two types of connected components.

\begin{lemma}\label{Type 1}
	There are $o(u_1)o(u_2)$ connected components in $\Delta$ for each connected component of $\Gamma_{u_1u_2}$ of Type $1$. Hence, Type $1$ connected components contribute
	$o(u_1)o(u_2)\beta_1^{(2)}(A_\Gamma)$ dimensions to $\mathrm{Der}\bigl(\ker(\varphi),\mathcal{D}(\ker(\varphi))\bigr)$.
\end{lemma}

\begin{proof}
	Let $\Omega$ be a connected component of Type $1$. Then $u_1$ and $u_2$ are not adjacent to any vertex of $\Omega$. Therefore, each vertex $v\in V(\Omega)$ gives rise to $o(u_1)o(u_2)$ distinct vertices in $\Delta$. Moreover, for each $0\leq i_1<o(u_1)$ and $\leq i_2<o(u_2)$,
	the set $\{v_{i_1,i_2}\mid v\in V(\Omega)\}$ induces a connected component of $\Delta$. Thus, $\Omega$ gives rise to $o(u_1)o(u_2)$ connected components of $\Delta$.
\end{proof}

\begin{lemma}\label{Type 3}
	There are $o(u_1)$ connected components in $\Delta$ for each connected component of $\Gamma_{u_1u_2}$ of Type $3$. Hence, Type $3$ connected components contribute
	$o(u_1)\bigl(\beta_1^{(2)}(A_{\Gamma_{u_2}})-\beta_1^{(2)}(A_\Gamma)\bigr)$
	dimensions to $\mathrm{Der}\bigl(\ker(\varphi),\mathcal{D}(\ker(\varphi))\bigr)$.
\end{lemma}

\begin{proof}
	Let $\Omega$ be a connected component of Type $3$. Then there exists $w\in V(\Omega)\cap\lk(u_2)$.
	For each $0\leq i_1<o(u_1)$, the vertices	$w_{i_1,0},\ldots,w_{i_1,o(u_2)-1}$
	are identified, so they give rise to $o(u_1)$ distinct vertices in $\Delta$. Moreover, for each $0\leq i_1<o(u_1)$, the set $\{v_{i_1,i_2}\mid v\in V(\Omega),\ 0\leq i_2<o(u_2)\}$ induces a connected component of $\Delta$. Thus, $\Omega$ gives rise to $o(u_1)$ connected components of $\Delta$.
\end{proof}
Now, let us move to the study of Type $4$ connected components.
\begin{lemma}
	For each connected component of Type $4$ in $\Gamma_{u_1u_2}$ there is one connected component in $\Delta$. 
\end{lemma}
\begin{proof}
	Let $v_{i_1,i_2},w_{i_3,i_4}$ be two vertices in $\Delta$ such that $v$ and $w$ lie in the same Type $4$ connected component $\Omega$ of $\Gamma_{u_1u_2}$. Let $r\in V(\Omega)\cap\lk(u_1)$ and $s\in V(\Omega)\cap\lk(u_2)$, then we can construct a path from $v_{i_1,i_2}$ to $w_{i_3,i_4}$ as follows:
	$$v_{i_1,i_2}\rightarrow s_{i_1,i_2}=s_{i_1,0}\rightarrow r_{i_1,0}=r_{i_3,0}\rightarrow s_{i_3,0}=s_{i_3,i_4}\rightarrow w_{i_3,i_4}$$
	and the claim follows.
\end{proof}
\begin{proposition}\label{Type4}
	Let $f\in \mathrm{Der}\bigl(\ker(\varphi),\mathcal{D}(\ker(\varphi))\bigr)$. Assume that $v_{i_1,i_2}$ and $v_{i_1+1,i_2}$ lie in the same connected component of $\Delta$ and that $\{v,u_1\}\in E(\Gamma)$. Then $(1-X_{i_1,i_2})^{-1}f(X_{i_1,i_2})=(1-v_{i_1,i_2})^{-1}f(v_{i_1,i_2})$. As a consequence Type $4$ connected components plus all generators of the form $X_{i,j}$ contribute to
	$1$ dimension to $\mathrm{Der}\bigl(\ker(\varphi),\mathcal{D}(\ker(\varphi))\bigr)$.
\end{proposition}

\begin{proof}
	Since $v_{i_1,i_2}$ and $v_{i_1+1,i_2}$ lie in the same connected component of $\Delta$, Lemma~\ref{Derivation Lemma}.1 gives
	$$(1-v_{i_1,i_2})^{-1}f(v_{i_1,i_2})=(1-v_{i_1+1,i_2})^{-1}f(v_{i_1+1,i_2})\Rightarrow
	f(v_{i_1+1,i_2})=(1-v_{i_1+1,i_2})(1-v_{i_1,i_2})^{-1}f(v_{i_1,i_2}).$$

	On the other hand, the relation $v_{i_1,i_2}=v_{i_1+1,i_2}^{X_{i_1,i_2}}$
	gives
	$$f(v_{i_1,i_2}X_{i_1,i_2})=f(X_{i_1,i_2}v_{i_1+1,i_2})\Rightarrow f(v_{i_1,i_2})+v_{i_1,i_2}f(X_{i_1,i_2})=f(X_{i_1,i_2})+X_{i_1,i_2}f(v_{i_1+1,i_2}).$$
	Substituting the expression for $f(v_{i_1+1,i_2})$ obtained above and rearranging, we obtain
	$$(1-v_{i_1,i_2})f(X_{i_1,i_2})=f(v_{i_1,i_2})+X_{i_1,i_2}(1-v_{i_1+1,i_2})(1-v_{i_1,i_2})^{-1}f(v_{i_1,i_2}).$$
	
	Since $X_{i_1,i_2}=(1-v_{i_1,i_2})X_{i_1,i_2}(1-v_{i_1+1,i_2})^{-1},$ we have
	$$(1-v_{i_1,i_2})f(X_{i_1,i_2})=f(v_{i_1,i_2})+(1-v_{i_1,i_2})X_{i_1,i_2}(1-v_{i_1,i_2})^{-1}f(v_{i_1,i_2}).$$
	Rearranging gives $(1-X_{i_1,i_2})^{-1}f(X_{i_1,i_2})=(1-v_{i_1,i_2})^{-1}f(v_{i_1,i_2})$, as required.	
	
	Since $u_1$ and $u_2$ lie in the same connected component of $\Gamma$ there is at least one connected component of Type $4$. Let $v_{i_1,i_2},w_{i_3,i_4}$ be two vertices in $\Delta$ such that $v$ and $w$ lie in Type $4$ connected components of $\Gamma$. If $v$ and $w$ lie in the same connected component of $\Gamma_{u_1,u_2}$ then $v_{i_1,i_2}$ and $w_{i_3,i_4}$ lie in the same connected component of $\Delta$, so by  Lemma~\ref{Derivation Lemma}.1 $(1-v_{i_1,i_2})^{-1}f(v_{i_1,i_2})=(1-w_{i_3,i_4})^{-1}f(w_{i_3,i_4})$. $v$ and $w$ lie in distinct connected component of $\Gamma_{u_1,u_2}$ then by the above discussion $(1-v_{i_1,i_2})^{-1}f(v_{i_1,i_2})=(1-X_{i_5,i_6})^{-1}f(X_{i_5,i_6})=(1-w_{i_3,i_4})^{-1}f(w_{i_3,i_4})$ for any $0<i_5<o(u_1)$ and  $0<i_6<o(u_2)$. Hence $(1-a)^{-1}f(a)$ is a constant value on the following set:
	 $$\lbrace v_{i_1,i_2}\mid 	0\leq i_k<o(u_k)\text{ with }k=1,2\text{ and }v\text{ is of Type }4\rbrace \cup\lbrace X_{i_1,i_2}\mid 	0<i_k<o(u_k)\text{ with }k=1,2\rbrace $$
	 Hence this set of generators contribute to $1$ dimension to $\mathrm{Der}\bigl(\ker(\varphi),\mathcal{D}(\ker(\varphi))\bigr)$.
\end{proof}
Finally, let us consider Type 2 connected components.
\begin{lemma}\label{Type 2}
	Type $2$ connected components contribute $o(u_2)\bigl(\beta_1^{(2)}(A_{\Gamma_{u_1}})-\beta_1^{(2)}(A_\Gamma)\bigr)$
	dimensions to $\mathrm{Der}\bigl(\ker(\varphi),\mathcal{D}(\ker(\varphi))\bigr)$.
\end{lemma}

\begin{proof}
	Let $\Omega$ be a connected component of Type $2$, and let
	$v\in V(\Omega)\cap\lk(u_1)$. Then $v_{i_1,i_2}X_{i_1,i_2}=X_{i_1,i_2}v_{i_1+1,i_2}$ for $0\leq i_k<o(u_k)$, $k=1,2$. If $f\in\mathrm{Der}\bigl(\ker(\varphi),\mathcal{D}(\ker(\varphi))\bigr)$, then
	$$f(v_{i_1,i_2})+v_{i_1,i_2}f(X_{i_1,i_2})=f(X_{i_1,i_2})+X_{i_1,i_2}f(v_{i_1+1,i_2}).$$
	Thus, the value $f(v_{i_1+1,i_2})$ is completely determined by
	$f(X_{i_1,i_2})$ and $f(v_{i_1,i_2})$. By Proposition~\ref{Type4}, the value of $f(X_{i_1,i_2})$ is completely determined by the Type $4$ connected components. Hence, for each fixed $i_2$, the value $f(v_{0,i_2})$ determines the values
	$f(v_{i_1,i_2})$ for all
	$i_1\in{1,\dots,o(u_1)-1}$.
	Therefore, all vertices of $\Delta$ arising from $\Omega$ are determined by the values
	$f(v_{0,i_2})$ for
	$i_2\in{0,\dots,o(u_2)-1}$.
	Thus, each Type $2$ connected component contributes $o(u_2)$ dimensions to
	$\mathrm{Der}\bigl(\ker(\varphi),\mathcal{D}(\ker(\varphi))\bigr)$.
	Since the number of Type $2$ connected components is
	$\beta_1^{(2)}(A_{\Gamma_{u_1}})
	-\beta_1^{(2)}(A_\Gamma)$, the claimed contribution follows.
\end{proof}
Combining all these results we get the computation for the first $\ell^2$-Betti number of this family of groups.
\begin{theorem}
	Let $\mathcal{G}_\Gamma$ be a graph product of cyclic groups with
	$V(\Gamma^{\infty})={u,v}$, and assume that $u$ and $v$ are not adjacent. Then
	$$\beta_1^{(2)}(\mathcal{G}_\Gamma)=\left(1-\frac{1}{o(u)}-\frac{1}{o(v)}\right)\beta_1^{(2)}(A_\Gamma)+\frac{1}{o(u)}\beta_1^{(2)}(A_{\Gamma_u})+\frac{1}{o(v)}\beta_1^{(2)}(A_{\Gamma_v}).$$
\end{theorem}

\begin{proof}
	First suppose that $u$ and $v$ lie in the same connected component of $\Gamma$. Then the result follows by adding the contributions to the dimension of
	$\mathrm{Der}\bigl(\ker(\varphi),\mathcal{D}(\ker(\varphi))\bigr)$
	from the four types of connected components described in Lemmas~\ref{Type 1}, \ref{Type 2}, \ref{Type 3}, and Proposition~\ref{Type4}. Now suppose that $u$ and $v$ lie in different connected components. We can then write $\mathcal{G}_\Gamma=\mathcal{G}_{\Gamma^1}\ast\mathcal{G}_{\Gamma^2}$, where $u\in V(\Gamma^1)$ and $v\in V(\Gamma^2)$. Assume that $\lvert V(\Gamma^1)\rvert,\lvert V(\Gamma^2)\rvert>1$, since the other two cases follow similarly. By \cite[Theorem 4.15(ii)]{Kammeyer},
	$$
	\beta_1^{(2)}(\mathcal{G}_\Gamma)=\beta_1^{(2)}(\mathcal{G}_{\Gamma^1})+\beta_1^{(2)}(\mathcal{G}_{\Gamma^2})+1.$$
	By Theorem~\ref{BettiClique}, we have
	$$\beta_1^{(2)}(\mathcal{G}_\Gamma)
		=\frac{1}{o(u)}\left((o(u)-1)\beta_1^{(2)}(A_{\Gamma^1})+\beta_1^{(2)}(A_{\Gamma^1_u})\right)+\frac{1}{o(v)}\left((o(v)-1)\beta_1^{(2)}(A_{\Gamma^2})+\beta_1^{(2)}(A_{\Gamma^2_v})\right)+1.
	$$
	On the other hand, since the first $\ell^2$-Betti number of a RAAG is given by the number of connected components of its defining graph minus $1$, we have
	$$
	\beta_1^{(2)}(A_\Gamma)
	=
	\beta_1^{(2)}(A_{\Gamma^1})
	+
	\beta_1^{(2)}(A_{\Gamma^2})
	+
	1,
	$$
	$$
	\beta_1^{(2)}(A_{\Gamma_u})
	=
	\beta_1^{(2)}(A_{\Gamma^1_u})
	+
	\beta_1^{(2)}(A_{\Gamma^2})
	+
	1,
	$$
	and
	$$
	\beta_1^{(2)}(A_{\Gamma_v})
	=
	\beta_1^{(2)}(A_{\Gamma^1})
	+
	\beta_1^{(2)}(A_{\Gamma^2_v})
	+
	1.
	$$
	Substituting these identities into the preceding expression for
	$\beta_1^{(2)}(\mathcal{G}_\Gamma)$ yields the desired formula.\end{proof}

\begin{proof}[Theorem \ref{Main3}]
	If $u$ and $v$ are adjacent, the result follows from Theorem~\ref{BettiClique}. Suppose now that $u$ and $v$ are not adjacent. A direct count of the connected components yields
	$$\beta_1^{(2)}(A_{\Gamma_{uv}})=\beta_1^{(2)}(A_{\Gamma_u})+\beta_1^{(2)}(A_{\Gamma_v})-\beta_1^{(2)}(A_\Gamma).$$
	Combining this identity with the formula from the previous theorem, we obtain the desired formula.
\end{proof}

\begin{remark}
	In the setting of the above theorem, suppose in addition that $\mathcal{G}_\Gamma$ is two-dimensional, that is, that $\Gamma$ contains no triangles. Then $\beta_i^{(2)}(\mathcal{G}_\Gamma)=0$ for all $i\geq 3$, and hence the second $\ell^2$-Betti number can be determined from the Euler characteristic. By \cite{Chriswell}, if $\deg(w)$ counts the number of vertices that are adjacent to $w$ in $\Gamma$: $$\chi(\mathcal{G}_\Gamma)=1-\lvert V(\Gamma)\rvert+\lvert E(\Gamma)\rvert+\frac{1-\deg(u)}{o(u)}+\frac{1-\deg(v)}{o(v)}.$$
	Consequently, $\beta_2^{(2)}(\mathcal{G}_\Gamma)=\beta_1^{(2)}(\mathcal{G}_\Gamma)+\chi(\mathcal{G}_\Gamma)$.
	 As a concrete example, if $\Gamma$ is a cycle with $n$ vertices, then
	$$\beta_i^{(2)}(\mathcal{G}_\Gamma)
	=
	\begin{dcases}
		1-\dfrac{1}{o(u)}-\dfrac{1}{o(v)}, & \text{if } i=2,\\[6pt]
		0, & \text{otherwise}.
	\end{dcases}$$
\end{remark}

 \section{$\Sigma$-Invariants of $\mathcal{G}_\Gamma$}\label{SigmaSection}
We begin by defining the conditions appearing in Theorem~\ref{Main3}.
\begin{definition}
	Let $\mathcal{G}_\Gamma$ be a RAAG and let $\chi:\mathcal{G}_\Gamma\to\mathbb{R}$ be a character.
	\begin{itemize}
		\item $\chi$ satisfies the \textbf{homological $n$-link condition} if $\widehat{\mathrm{lk}_{\mathrm{Liv}^\chi}(\Delta)}$ is $(n-\abs{V(\Delta)}-1)$-acyclic for every clique $\Delta\subset\Gamma\setminus\mathrm{Liv}^\chi$ with $\abs{V(\Delta)}\leq n$.
		\item $\chi$ satisfies the \textbf{homotopical $n$-link condition} if $\widehat{\mathrm{lk}_{\mathrm{Liv}^\chi}(\Delta)}$ is $(n-\abs{V(\Delta)}-1)$-acyclic for every clique $\Delta\subset\Gamma\setminus\mathrm{Liv}^\chi$ with $\abs{V(\Delta)}\leq n$ and $\widehat{\mathrm{Liv}^\chi}$ is $(n-1)$-connected.	
		\item $\chi$ satisfies the \textbf{strong homotopical $n$-link condition}  if $\widehat{\mathrm{lk}_{\mathrm{Liv}^\chi}(\Delta)}$ is $(n-\abs{V(\Delta)}-1)$-connected for every clique $\Delta\subset\Gamma\setminus\mathrm{Liv}^\chi$ with $\abs{V(\Delta)}\leq n$.
	\end{itemize}
\end{definition}
The first two conditions were used by Meier, Meinert, and VanWyk to compute the BNSR-invariants of RAAGs.
\begin{theorem}[\cite{Meier-Meinert-VanWyk 1998} Meier-Meinert-VanWyk, 1998]\label{SigmaRAAG}
	Let $A_\Gamma$ be a RAAG and let $\chi:A_\Gamma\to\mathbb{R}$ be a character. Then:
	\begin{enumerate}
		\item $[\chi]\in\Sigma^n(A_\Gamma,\mathbb{Z})$ if and only if $\chi$ satisfies the homological $n$-link condition.
		\item $[\chi]\in\Sigma^n(A_\Gamma)$ if and only if $\chi$ satisfies the homotopical $n$-link condition.
	\end{enumerate}
\end{theorem}
The aim of this section is to compute the $\Sigma$-invariants of graph products of cyclic groups. The case $n=1$ has already been established.
\begin{theorem}[\cite{Meinert 1995} Meinert, 1995]\label{Sigma1}
	Let $\mathcal{G}_\Gamma$ be a graph product of cyclic groups. Then
	\[
	\Sigma^1(\mathcal{G}_\Gamma)=\{[\chi]\in S(\mathcal{G}_\Gamma)\mid\mathrm{Liv}^\chi\text{ is connected and dominant}\}.
	\]
	A subgraph $\Gamma'\subset\Gamma$ is called \textbf{dominant} if, for every vertex $v\in V(\Gamma\setminus\Gamma')$, there exists a vertex $w\in V(\Gamma')$ such that $\{v,w\}\in E(\Gamma)$.
\end{theorem}

This follows because the $1$-link condition is equivalent to the living subgraph being connected and dominant. 

To prove one of the implications in Theorem~\ref{Main3}, we will use an argument analogous to that in \cite{Escartin-Perez 2024}, where the $\Sigma$-invariants of Artin groups satisfying the $K(\pi,1)$-conjecture are computed.

Let $G$ be a group, $\mathcal{P}$ a poset of subgroups of $G$ and $\mathcal{H}\subset\mathcal{P}$ a subposet. Assume we have a bounded height function $h:\mathcal{P}\setminus\mathcal{H}\to\mathbb{Z}_{\geq 0}$ such that $h(A)<h(B)$ for all $A\varsubsetneq B$. Given $S\in\mathcal{P}$, denote by $\mathrm{Del}^S$ to the geometric realization of the $S$-coset poset associated to $\lbrace T\in\mathcal{P}\mid T\nleq S\rbrace$ and consider the poset $\mathcal{J}_S=\lbrace T\in\mathcal{H}\mid S\subset T\rbrace$. If $C(\mathcal{P})$ and $C(\mathcal{H})$ denote the geometric realizations of the coset poset of $\mathcal{P}$ and $\mathcal{H}$ respectively we have:
\begin{lemma}[\cite{Escartin-Perez 2024} Lemma 2.10]\label{Conectivity Lemma} With the same notation as above assume that $C(\mathcal{P})$ is $(n-1)$-connected and that for all $S\in\mathcal{P}\setminus\mathcal{H}$ with $h(S)=s\leq n$ we have $\mathrm{Del}^S$ is $(s-2)$-connected, Then:
	\begin{enumerate}
		\item[(i)] If $\abs{\mathcal{J}_S}$ is $(n-s-1)$-connected, $C(\mathcal{H})$ is $(n-1)$-connected.
		\item[(ii)] If $\abs{\mathcal{J}_S}$ is $(n-s-1)$-acyclic, $C(\mathcal{H})$ is $(n-1)$-acyclic.
	\end{enumerate}
\end{lemma}

In our setting we will consider $G=\mathcal{G}_\Gamma$, $\mathcal{P}=\lbrace \mathcal{G}_\Delta\mid \Delta\subset\Gamma\textrm{ complete }\rbrace$, $\mathcal{H}=\lbrace\mathcal{G}_\Delta\in\mathcal{P}\mid\chi(G_\Delta)\neq0\rbrace$ and $h(\mathcal{G}_\Delta)=\abs{V(\Delta)}$. The connectivity of $\mathrm{Del}^S$ for $S\in\mathcal{P}\setminus\mathcal{H}$ follows in the same way as in \cite[Lemma 2.11]{Escartin-Perez 2024}. For the convenience of the reader we give an sketch of the proof.
\begin{lemma}\label{DeligneLemma}
	Let $S=\mathcal{G}_\Delta$ be a graph product of cyclic groups based on a complete graph $\Delta$ with $\abs{V(\Delta)}=s\neq0$. Then, $\mathrm{Del}^S$ is homotopically equivalent to a wedge of $(s-1)$-spheres, so it is $(s-2)$-connected.
\end{lemma}
\begin{proof} For each $v\in V(\Delta)$ and $g\in S$ let $gX_v$ be the subcomplex of $\mathrm{Del}^S$ induced by $\lbrace g\mathcal{G}_\Omega\mid\Omega\subset\Delta_v\rbrace$. The set $\lbrace gX_v\mid g\in S,v\in V(\Delta)\rbrace$ is a covering of $\mathrm{Del}^S$. One can see that any set of pairwise distinct $g_0X_{v_0},\dots, g_k X_{v_k}$ has empty intersection if two of the $v_i's$ are equal and that whenever such a set has a non-empty intersection the intersection of a subcomplex with a maximal element, hence it is contractible. Thus, $\mathrm{Del}^S$ is homotopy equivalent to the nerve of the covering (cf. \cite{Brown 2012} Chapter VII Theorem 4.4) and by (cf. \cite{Deligne 1972} Theorem 2.15) this complex is a wedge of $(s-1)$-spheres.
\end{proof}
\begin{proposition}
	Let $\mathcal{G}_\Gamma$ be a graph product of cyclic groups and $\chi:\mathcal{G}_\Gamma$ a non-zero character. Then:
	\begin{itemize}
		\item[(i)] If $\chi$ satisfies the homological $n$-link condition then $[\chi]\in\Sigma^n(\mathcal{G}_\Gamma,\mathbb{Z}) $.
		\item[(ii)]If $\chi$ satisfies the strong homotopical $n$-link condition then $[\chi]\in\Sigma^n(\mathcal{G}_\Gamma)$.
	\end{itemize}
\end{proposition}
\begin{proof}
	Let us prove the homological one first. We first need to prove that $C(\mathcal{H})$ is $(n-1)$-acyclic for which we use Lemma \ref{Conectivity Lemma}. $C(\mathcal{P})$ is contractible by \cite{Davis 1998} and $\mathrm{Del}^S$ is $(s-2)$-connected by Lemma \ref{DeligneLemma}, so we just need to prove that $\abs{\mathcal{J}_S}$ is $(n-s-1)$-connected for each $S=\mathcal{G}_\Delta\in\mathcal{P}\setminus\mathcal{H}$ and with $\abs{V(S)}=s$. Let $\mathcal{L}^\chi_\Delta$ be the poset of simplices of $\widehat{\mathrm{lk}_{\mathrm{Liv}^\chi}(\Delta)}$ and define the map:
	$$\begin{aligned}
		\tau :\mathcal{L}^\chi_\Delta & \longrightarrow \mathcal{J}_S \\
		Y & \longmapsto Y\cup\Delta
	\end{aligned}$$
	one sees that for each $X\in\mathcal{J}_S$ $\tau^{-1}_{\leq}(X)=\lbrace Y\in \mathcal{L}^\chi_\Delta\mid Y\cup\Delta\subset X\rbrace$ which has $X\setminus\Delta$ as a maximal element, thus it is contractible. Then, by Quillen's Lemma (cf. \cite{Benson} Lemma 6.5.2) $\tau$ induces an homotopy equivalence between the geometric realizations. Hence, the homological $n$-link conditions imply that $\abs{\mathcal{L}^\chi_\Delta}$, and so $\abs{\mathcal{J}_S}$ is $(n-s-1)$-acyclic. Hence $C(\mathcal{H})$ is an $(n-1)$-acyclic space with stabilizers of the form $g\mathcal{G}_\Delta g^{-1}$ with $g\in\mathcal{G}_\Gamma$ and $\mathcal{G}_\Delta\in\mathcal{H}$. Applying Theorem \ref{teo:MMVW} the claim follows. The same proof works for the homotopical case.
\end{proof}

For the converse we begin with the following lemma.
\begin{lemma}\label{LivLemma}
	Let $\chi:\mathcal{G}_\Gamma\to\mathbb{R}$ be a character such that $[\chi]\in\Sigma^n(\mathcal{G}_\Gamma,\mathbb{Z})$. Then:
	\begin{enumerate}
		\item[(i)] $[\restr{\chi}{\mathrm{Liv}^\chi}]\in\Sigma^n(A_{\mathrm{Liv}^\chi},\mathbb{Z})$.
		\item[(ii)] For every $v\in V(\Gamma\setminus\mathrm{Liv}^\chi)$, $\restr{\chi}{\mathcal{G}_{\mathrm{lk}(v)}}\neq0$ and $[\restr{\chi}{\mathcal{G}_{\mathrm{lk}(v)}}]\in\Sigma^{n-1}(\mathcal{G}_{\mathrm{lk}(v)},\mathbb{Z})$.
	\end{enumerate}
\end{lemma}
\begin{proof}
	Consider the retraction $\rho:\mathcal{G}_\Gamma\longrightarrow A_{\mathrm{Liv}^\chi}$, defined by $\rho(v)=v$ if $\chi(v)\neq0$ and $\rho(v)=1$ if $\chi(v)=0$.
	This yields the semidirect product decomposition $\mathcal{G}_\Gamma=\ker(\rho)\rtimes A_{\mathrm{Liv}^\chi}$. Hence (i) follows from \cite[Theorem 5.3]{Meier-Meinert-VanWyk 1998}.
	
	Similarly, let $\varphi:\mathcal{G}_\Gamma\longrightarrow\mathcal{G}_{\Gamma_v}$ be the natural retraction. Then, $[\restr{\chi}{\mathcal{G}_{\Gamma_v}}]\in\Sigma^n(\mathcal{G}_{\Gamma_v},\mathbb{Z})$. Moreover, from the standard presentation of graph products there is an HNN-product decomposition $\mathcal{G}_\Gamma=\mathcal{G}_{\Gamma_v}\ast_{\mathcal{G}_{\mathrm{lk}(v)},v}$. By Theorem \ref{Sigma1} it follows that $\restr{\chi}{\mathcal{G}_{\mathrm{lk}(v)}}\neq0$. Therefore, (ii) follows from \cite[Theorem 5.2]{Meier-Meinert-VanWyk 1998}.
\end{proof} 
The homotopical analogue follows by the same argument.
\begin{lemma}\label{LivLemma2}
	Let $\chi:\mathcal{G}_\Gamma\to\mathbb{R}$ be a character such that $[\chi]\in\Sigma^n(\mathcal{G}_\Gamma)$. Then:
	\begin{enumerate}
		\item[(i)] $[\restr{\chi}{\mathrm{Liv}^\chi}]\in\Sigma^n(A_{\mathrm{Liv}^\chi})$.
		\item[(ii)] For every $v\in V(\Gamma\setminus\mathrm{Liv}^\chi)$, $\restr{\chi}{\mathcal{G}_{\mathrm{lk}(v)}}\neq0$ and $[\restr{\chi}{\mathcal{G}_{\mathrm{lk}(v)}}]\in\Sigma^{n-1}(\mathcal{G}_{\mathrm{lk}(v)},\mathbb{Z})$.
	\end{enumerate}
\end{lemma}
Using these lemmas, we can prove the converse implication in Theorem~\ref{Main3}.
\begin{proposition}\begin{enumerate}
		\item If $[\chi]\in\Sigma^n(\mathcal{G}_\Gamma,\mathbb{Z})$, then $\chi$ satisfies the homological $n$-link condition.
		\item If $[\chi]\in\Sigma^n(\mathcal{G}_\Gamma)$, then $\chi$ satisfies the homotopical $n$-link condition.
	\end{enumerate}
\end{proposition}

\begin{proof}
	We argue by induction on $n$. If $n=1$, the result follows from Theorem~\ref{Sigma1}. Assume now that the statement holds for $n-1$.
	
	By Lemma~\ref{LivLemma}(i), $[\restr{\chi}{\mathrm{Liv}^\chi}]\in\Sigma^n(A_{\mathrm{Liv}^\chi},\mathbb{Z})$, and therefore, by Theorem~\ref{SigmaRAAG}, $\widehat{\mathrm{lk}_{\mathrm{Liv}^\chi}(\emptyset)}=\widehat{\mathrm{Liv}^\chi}$
	is $(n-1)$-acyclic.
	
	Now let $\Delta\subset\Gamma\setminus\mathrm{Liv}^\chi$ be a clique with $1\leq\abs{V(\Delta)}\leq n-1$,
	and write $V(\Delta)=\{v_1,\ldots,v_r\}$. Applying Lemma~\ref{LivLemma}(ii) repeatedly, we obtain
	\[
	\left[\restr{\chi}
	{\mathcal{G}_{\mathrm{lk}(v_1)\cap\cdots\cap\mathrm{lk}(v_r)}}\right]
	\in
	\Sigma^{n-r}
	\left(
	\mathcal{G}_{\mathrm{lk}(v_1)\cap\cdots\cap\mathrm{lk}(v_r)},
	\mathbb{Z}
	\right).
	\]
	By the induction hypothesis, $\widehat{
		\mathrm{lk}_{\mathrm{Liv}^\chi\cap
			\mathrm{lk}(v_1)\cap\cdots\cap\mathrm{lk}(v_r)}
		(\emptyset)}
	=
	\widehat{\mathrm{lk}_{\mathrm{Liv}^\chi}(\Delta)}$
	is $(n-1-\abs{V(\Delta)})$-acyclic.
	
	Finally, let $\Delta\subset\Gamma\setminus\mathrm{Liv}^\chi$ be a clique with $\abs{V(\Delta)}=n$,
	and write $V(\Delta)=\{v_1,\ldots,v_n\}$.
	Again by Lemma~\ref{LivLemma},
	$$\left[\restr{\chi}{\mathcal{G}_{\mathrm{lk}(v_1)\cap\cdots\cap\mathrm{lk}(v_{n-1})}}
	\right]\in\Sigma^1\left(\mathcal{G}_{\mathrm{lk}(v_1)\cap\cdots\cap\mathrm{lk}(v_{n-1})},\mathbb{Z}\right).$$
	Hence, $\mathrm{Liv}^\chi\cap\mathrm{lk}(v_1)\cap\cdots\cap\mathrm{lk}(v_{n-1})$ is dominant in $\mathrm{lk}(v_1)\cap\cdots\cap\mathrm{lk}(v_{n-1})$.
	Therefore there exists $u\in
	V(\mathrm{lk}(v_1)\cap\cdots\cap\mathrm{lk}(v_{n-1}))$
	such that $\{u,v_n\}\in E(\Gamma)$ and $\chi(u)\neq 0$. It follows that $u\in\bigcap_{i=1}^n\mathrm{lk}(v_i)$,
	and consequently $\widehat{\mathrm{lk}_{\mathrm{Liv}^\chi}(\Delta)}$
	is non-empty, i.e. $(-1)$-acyclic.
	Thus $\chi$ satisfies the homological $n$-link condition.The homotopical part follows analogously, using the same argument together with Lemma~\ref{LivLemma2} instead of Lemma~\ref{LivLemma}. 
\end{proof}

\section{$\ell^2$-Betti numbers of \texorpdfstring{$\mathrm{Aut}(\mathcal{G}_\Gamma)$}{Aut(Gamma)} and \texorpdfstring{$\mathrm{Out}(\mathcal{G}_\Gamma)$}{Out(Gamma)}}
\subsection{Zeroth $\ell^2$-Betti number}

In this section we determine precisely when the groups $\mathrm{Aut}(\mathcal{G}_\Gamma)$ and $\mathrm{Out}(\mathcal{G}_\Gamma)$ are finite. This, in turn, completely characterises the vanishing of the zeroth $\ell^2$-Betti number of these groups. First, let us consider the case of $\operatorname{Aut}(\mathcal{G}_\Gamma)$.
\begin{proposition}\label{FiniteAut}
Let $\mathcal{G}_\Gamma$ be a graph product of cyclic groups. Then $\mathrm{Aut}(\mathcal{G}_\Gamma)$ is finite if and only if $\Gamma$ is complete and $|V(\Gamma^\infty)|\leq 1$, where $V(\Gamma^\infty)$ denotes the set of vertices corresponding to infinite cyclic vertex groups.
\end{proposition}

\begin{proof}
If $\mathrm{Aut}(\mathcal{G}_\Gamma)$ is finite, then its subgroup of inner automorphisms
$\mathrm{Inn}(\mathcal{G}_\Gamma)=\mathcal{G}_\Gamma/\mathrm{Z}(\mathcal{G}_\Gamma)$
is also finite.
We know that $\mathcal{G}_\Gamma/\mathrm{Z}(\mathcal{G}_\Gamma)\cong \mathcal{G}_\Delta$, where $\Delta$ is the subgraph induced by the vertices $v\in V(\Gamma)$ with $\mathrm{st}(v)\neq \Gamma$.
If this quotient is finite, then $\Delta$ must be complete, which implies $\mathrm{st}(v)=\Gamma$ for all $v\in V(\Gamma)$. Hence, if $\mathcal{G}_\Gamma/\mathrm{Z}(\mathcal{G}_\Gamma)$ is finite then $\mathcal{G}_\Gamma=\mathrm{Z}(\mathcal{G}_\Gamma)$, i.e. $\Gamma$ is complete.

In this case, $\mathcal{G}_\Gamma\cong \mathbb{Z}^k\times G$, where $G$ is a finite group.
If $k=0$, then $\mathrm{Aut}(\mathcal{G}_\Gamma)$ is clearly finite. Suppose $k>0$. Then $\mathrm{Aut}(\mathbb{Z}^k)$ embeds as a finite-index subgroup of $\mathrm{Aut}(\mathbb{Z}^k\times G)$. Thus $\mathrm{Aut}(\mathcal{G}_\Gamma)$ is finite if and only if $\mathrm{Aut}(\mathbb{Z}^k)$ is finite, which occurs precisely when $k=1$. This proves the claim.
\end{proof}

For the outer case, observe that, by Theorem~\ref{GeneratorsAUT}, there are two families of infinite-order generators of $\operatorname{Aut}(\mathcal{G}_\Gamma)$:
	\begin{itemize}
		\item Transvections corresponding to vertices $u,v$ such that $\mathrm{lk}(u)\subset\mathrm{st}(v)$ and $o(u)=o(v)=\infty$.
		\item Partial conjugations $\pi^v_C$ with $o(v)=\infty$.
	\end{itemize}

However, excluding these generators is not sufficient to guarantee finiteness: it is possible for two finite-order partial conjugations to have a product of infinite order.
To address this, we introduce a condition that encodes the commutation of partial conjugations.

\begin{definition}
Let $\Gamma$ be a simplicial graph and $u,v\in V(\Gamma)$ be non-adjacent vertices.
The pair $(u,v)$ forms a \textbf{separating intersection of links (SIL)} if there exists a connected component of $\Gamma\setminus(\mathrm{lk}(u)\cap\mathrm{lk}(v))$ which contains neither $u$ nor $v$.
\end{definition}

Let $\overline{\mathcal{G}}_\Gamma=\mathbb{Z}^n\times G$ denote the abelianisation of $\mathcal{G}_\Gamma$, where $G$ is finite.
There is a canonical map
$$\rho:\mathrm{Aut}(\mathcal{G}_\Gamma)\longrightarrow \mathrm{Aut}(\overline{\mathcal{G}}_\Gamma),$$
whose kernel is called the \textbf{Torelli subgroup} $\mathrm{IA}_\Gamma$.
Sale and Susse proved in \cite[Theorem 5.1]{Sale-Susse} that $\mathrm{IA}_\Gamma$ is generated by all partial conjugations and by commutator transvections $[\tau^v_w,\tau^u_w]$ whenever $w\leq u,v$ and $o(u)=\infty$.

We now show that if all transvections of $\mathrm{Aut}(\mathcal{G}_\Gamma)$ have finite order, then $\mathrm{IA}_\Gamma$ is a finite index subgroup.

\begin{proposition}\label{Finite index Torelli}
Let $\mathcal{G}_\Gamma$ be a graph product of cyclic groups such that whenever 
$\mathrm{lk}(u)\subset \mathrm{st}(v)$ and $o(u)=\infty$, we have $o(v)<\infty$. 
Then $\mathrm{IA}_\Gamma$ has finite index in $\mathrm{Aut}(\mathcal{G}_\Gamma)$.
\end{proposition}

\begin{proof}
Let $\mathrm{Aut}^1_\infty(\mathcal{G}_\Gamma)$ denote the subgroup of $\mathrm{Aut}(\mathcal{G}_\Gamma)$ generated by all partial conjugations and by those transvections $\tau^v_u$ with $u,v\in V(\Gamma)$ satisfying $\mathrm{lk}(u)\subset \mathrm{st}(v)$ and $o(u)=\infty$. 
By \cite[Corollary 2.8]{Sale-Susse}, this subgroup has finite index in $\mathrm{Aut}(\mathcal{G}_\Gamma)$ and contains $\mathrm{IA}_\Gamma$.

Under the given hypothesis, no transvection of infinite order occurs. 
Hence, by \cite[Lemma 5.3]{Sale-Susse}, the image $\rho(\mathrm{Aut}^1_\infty(\mathcal{G}_\Gamma))$ in $\mathrm{Aut}(\overline{\mathcal{G}}_\Gamma)$ is finite. 
Since $\mathrm{IA}_\Gamma=\ker(\rho|_{\mathrm{Aut}^1_\infty(\mathcal{G}_\Gamma)})$, it follows that $\mathrm{IA}_\Gamma$ is of finite index in $\mathrm{Aut}^1_\infty(\mathcal{G}_\Gamma)$, and consequently of finite index in $\mathrm{Aut}(\mathcal{G}_\Gamma)$.
\end{proof}

\begin{proposition}\label{FiniteOut}
Let $\mathcal{G}_\Gamma$ be a graph product of cyclic groups. Then $\mathrm{Out}(\mathcal{G}_\Gamma)$ is finite if and only if the following conditions hold:
\begin{enumerate}
    \item $\Gamma$ has no SILs;
    \item $\Gamma\setminus\mathrm{st}(v)$ is connected for every $v\in V(\Gamma^\infty)$;
    \item if $\mathrm{lk}(u)\subset\mathrm{st}(v)$ and $o(u)=\infty$, then $o(v)<\infty$.
\end{enumerate}
\end{proposition}

\begin{proof}
If any of these conditions fails, then $\mathrm{Out}(\mathcal{G}_\Gamma)$ is infinite.
Indeed, if $(u,v)$ is a SIL and $C_u,C_v$ denote the connected components of $\Gamma\setminus\mathrm{st}(u)$ and $\Gamma\setminus\mathrm{st}(v)$ containing $v$ and $u$ respectively, then $\pi_{C_u}^u\pi_{C_v}^v$ has infinite order (c.f. \cite[Section 4]
{Gutierrez-Piggott-Ruane}).
Similarly, violation of conditions (2) or (3) produces standard generators of infinite order.

Conversely, assume that all three conditions hold.
By Proposition~\ref{Finite index Torelli}, $\mathrm{IA}_\Gamma$ has finite index in $\mathrm{Aut}(\mathcal{G}_\Gamma)$.
Since $\Gamma$ has no SILs, it follows from \cite[Section 4]
{Gutierrez-Piggott-Ruane} that all partial conjugations commute.
Hence $\mathrm{IA}_\Gamma$ is a finite abelian group.
Therefore, $\mathrm{Out}(\mathcal{G}_\Gamma)$ contains a finite subgroup of finite index, and is itself finite.
\end{proof}
\subsection{The first $\ell^2$-Betti number of $\mathrm{Aut}(\mathcal{G}_\Gamma)$}

In this section we compute the first $\ell^2$-Betti number of $\mathrm{Aut}(\mathcal{G}_\Gamma)$ in the case where $\abs{V(\Gamma^{<\infty})}\leq 2$.

We begin with a couple of technical lemmas before proving Theorem~\ref{Main4}.

\begin{lemma}\label{LemmaDisconnected}
	Let $\mathcal{G}_\Gamma$ be a graph product of cyclic groups with $\Out(\mathcal{G}_\Gamma)$ finite. Assume that $\Gamma$ is connected and that there exists $v\in V(\Gamma)$ such that $o(v)<\infty$ and $\Gamma_v$ is disconnected. If there exists a connected component $C$ of $\Gamma_v$ lying entirely in $\Gamma^\infty$, then $\abs{V(C)}=1$ and $\Gamma_v$ has two connected component.
\end{lemma}

\begin{proof}
	Let $w\in V(C)$ be a vertex adjacent to $v$. Since $\Out(\mathcal{G}_\Gamma)$ is finite, the graph $\Gamma\setminus\st(w)$ is connected. Hence, $C=\st(w)\setminus\{v\}$, and the same holds for every $w\in V(C)$. Therefore, for every $u,w\in V(C)$ we have $\lk(u)\subseteq\st(w)$. This implies that  $\abs{V(C)}=1$, otherwise $\Out(\mathcal{G}_\Gamma)$ would be infinite. It also follows that $\Gamma_v$ splits in two connected components $C$ and $\Gamma\setminus(C\cup\lbrace v\rbrace)$.
\end{proof}

\begin{lemma}\label{Gamma connected}
Let $\mathcal{G}_\Gamma$ be a graph product of cyclic groups with $V(\Gamma^{<\infty})=\{v\}$. 
If $\mathrm{Out}(\mathcal{G}_\Gamma)$ is finite, then:
\begin{enumerate}
    \item either $\Gamma$ is connected, or $\mathcal{G}_\Gamma \cong \mathbb{Z}\ast \mathbb{Z}/o(v)\mathbb{Z}$;
    \item the subgraph $\Gamma_v$ is connected.
\end{enumerate}
\end{lemma}

\begin{proof}
For the first statement, suppose by contradiction that $\Gamma$ is disconnected. 
By \cite[Lemma~2.17]{Escartin}, since $\Gamma$ has no SILs, it must be the disjoint union of two complete subgraphs. 
The only such configuration with no infinite-order transvections corresponds to a disconnected graph with two vertices, and hence $\mathcal{G}_\Gamma \cong \mathbb{Z}\ast \mathbb{Z}/o(v)\mathbb{Z}$.

For the second statement, assume that $\Gamma_v$ is disconnected, with connected components $\Gamma_1,\dots,\Gamma_n$. By 1. $\Gamma$ is connected, so we can use Lemma \ref{LemmaDisconnected} and it follows that $\abs{V(\Gamma_1)}=\dots=\abs{V(\Gamma_n)}=1$. Therefore
$$\mathcal{G}_\Gamma \cong F_n\times \mathbb{Z}/o(v)\mathbb{Z},$$
which admits transvections of infinite order unless $n=1$. Hence, since  $\mathrm{Out}(\mathcal{G}_\Gamma)$ is finite it follows that $n=1$ and $\mathcal{G}_\Gamma \cong \mathbb{Z}\times \mathbb{Z}/o(v)\mathbb{Z}$, contradicting that $\Gamma_v$ is disconnected.
\end{proof}
\begin{proof}(Theorem \ref{Main4})
By Proposition~\ref{FiniteAut}, the group $\mathrm{Aut}(\mathcal{G}_\Gamma)$ is finite if and only if $\Gamma$ is a complete graph on one or two vertices. 
Hence we may assume that we are not in these cases.

Consider the short exact sequence
$$1 \longrightarrow \mathcal{G}_\Gamma/\mathrm{Z}(\mathcal{G}_\Gamma)
\longrightarrow \mathrm{Aut}(\mathcal{G}_\Gamma)
\longrightarrow \mathrm{Out}(\mathcal{G}_\Gamma)
\longrightarrow 1.$$
If $\mathcal{G}_\Gamma/\mathrm{Z}(\mathcal{G}_\Gamma)$ is finite, then $\mathcal{G}_\Gamma=\mathrm{Z}(\mathcal{G}_\Gamma)$, so 
$\mathcal{G}_\Gamma \cong \mathbb{Z}^k\times  \mathbb{Z}/o(v)\mathbb{Z}$. 
In this case $\mathrm{Aut}(\mathbb{Z}^k)$ embeds as a finite-index subgroup of $\mathrm{Aut}(\mathbb{Z}^k\times  \mathbb{Z}/o(v)\mathbb{Z})$, and 
$\beta_1^{(2)}(\mathrm{Aut}(\mathbb{Z}^k))\neq 0$ if and only if $k=2$.

If both $\mathcal{G}_\Gamma/\mathrm{Z}(\mathcal{G}_\Gamma)$ and $\mathrm{Out}(\mathcal{G}_\Gamma)$ are infinite, then by \cite[Theorem~7.2~(7)]{Luck} we have $\beta_1^{(2)}(\mathrm{Aut}(\mathcal{G}_\Gamma))=0$. 
Therefore, we may assume that $\mathrm{Out}(\mathcal{G}_\Gamma)$ is finite, so that $\mathcal{G}_\Gamma/\mathrm{Z}(\mathcal{G}_\Gamma)$ has finite index in $\mathrm{Aut}(\mathcal{G}_\Gamma)$.

If $\mathrm{Z}(\mathcal{G}_\Gamma)\neq 1$, then $\mathrm{Z}(\mathcal{G}_\Gamma)=\langle v\rangle$; otherwise, by condition~(3) of Proposition~\ref{FiniteOut}, $\mathrm{Out}(\mathcal{G}_\Gamma)$ would be infinite. 
In this case $\mathcal{G}_\Gamma/\mathrm{Z}(\mathcal{G}_\Gamma)\cong A_{\Gamma_v}$, which is a right-angled Artin group. 
By Lemma~\ref{Gamma connected}(2), the graph $\Gamma_v$ is connected, so $\beta_1^{(2)}(A_{\Gamma_v})=0$, and hence $\beta_1^{(2)}(\mathrm{Aut}(\mathcal{G}_\Gamma))=0$.

If $\mathrm{Z}(\mathcal{G}_\Gamma)=1$, then $\mathcal{G}_\Gamma$ has finite index in $\mathrm{Aut}(\mathcal{G}_\Gamma)$. 
By Theorem~\ref{BettiClique}, $\beta_1^{(2)}(\mathcal{G}_\Gamma)=0$ if and only if both $\Gamma$ and $\Gamma_v$ are connected. 
Thus, by Lemma~\ref{Gamma connected}, we deduce that $\beta_1^{(2)}(\mathrm{Aut}(\mathcal{G}_\Gamma))=0$ unless $\mathcal{G}_\Gamma\cong\mathbb{Z}\ast\mathbb{Z}/o(v)\mathbb{Z}$, completing the proof.
\end{proof}
To understand the case of two finite order vertices we need to introduce a family of graphs.
\begin{definition}
	Assume $\mathcal{G}_\Gamma$ is a graph product  of cyclic groups with $V(\Gamma^{<\infty})=\{u,v\}$ that $\Gamma,\Gamma_u,\Gamma_v$ are connected and that $\Gamma_{u,v}$ is disconnected. If the decomposition of $\Gamma_{u,v}$ in connected components is $C_1\sqcup\dots\sqcup C_n$ we say that $\Gamma$ is of \textbf{type $\mathcal{P}$} if at most one connected component $C_i$ satisfies $V(C_i)=1$ and every other component $C_j$ satisfies that the only infinite order generators of $\Out(A_{C_j\cup\{u,v\}})$, if any, are partial conjugations or transvections with respect to $u$ and $v$. 
\end{definition}
\begin{lemma}\label{Disconnected2}
	Let $\mathcal{G}_\Gamma$ be a graph product of cyclic groups with $V(\Gamma^{<\infty})=\{u,v\}$. 
	If $\mathrm{Out}(\mathcal{G}_\Gamma)$ is finite, then:
	\begin{enumerate}
		\item Either $\Gamma$ is connected, $\mathcal{G}_\Gamma\cong\mathbb{Z}\ast\left(\mathbb{Z}/o(u)\mathbb{Z}\times\mathbb{Z}/o(v)\mathbb{Z}\right)$ or $\mathcal{G}_\Gamma\cong\left(\mathbb{Z}\times\mathbb{Z}/o(u)\mathbb{Z}\right)\ast\left(\mathbb{Z}\times\mathbb{Z}/o(v)\mathbb{Z}\right)$.
		\item If $\Gamma$ is connected, either $\Gamma_v$ is connected or $\Gamma$ is one of the following graphs:
		$$\begin{tikzpicture}[main/.style = {draw, circle},node distance={15mm}] 
			\node[][main] (1)  {}; 
			\node[label={$v$}][main] (2) [right of=1] {};
			\node[][main] (3) [right of=2] {}; 
			\node[label={$u$}][main] (4) [right of=3] {};
			\draw[-] (1) -- (2);
			\draw[-] (2) -- (3);
			\draw[-] (3) -- (4);
		\end{tikzpicture}~~~~~~~~
		\begin{tikzpicture}[main/.style = {draw, circle},node distance={15mm}] 
			\node[][main] (1)  {}; 
			\node[label={$v$}][main] (2) [right of=1] {};
			\node[][main] (3) [above right of=2] {}; 
			\node[label={$u$}][main] (4) [right of=2] {};
			\draw[-] (1) -- (2);
			\draw[-] (2) -- (3);
			\draw[-] (3) -- (4);
			\draw[-] (2) -- (4);
		\end{tikzpicture} $$
		$$\begin{tikzpicture}[main/.style = {draw, circle},node distance={15mm}] 
			\node[][main] (1)  {}; 
			\node[label={$v$}][main] (2) [right of=1] {};
			\node[label={$u$}][main] (3) [right of=2] {}; 
			\draw[-] (1) -- (2);
			\draw[-] (2) -- (3);
		\end{tikzpicture}~~~~~~~~\begin{tikzpicture}[main/.style = {draw, circle},node distance={15mm}] 
			\node[][main] (1)  {}; 
			\node[label={$v$}][main] (2) [right of=1] {};
			\node[label={$u$}][main] (3) [right of=2] {}; 
			\node[][main] (4) [right of=3] {};
			\draw[-] (1) -- (2);
			\draw[-] (2) -- (3);
			\draw[-] (3) -- (4);
		\end{tikzpicture}$$
		\item If $\Gamma,\Gamma_u,\Gamma_v$ are connected and $\Gamma_{u,v}$ is disconnected then $\Gamma$ is of type $\mathcal{P}$ and $\mathrm{Z}(\mathcal{G}_\Gamma)=1$.
	\end{enumerate}
\end{lemma}
\begin{proof}
	For the first statement, suppose by contradiction that $\Gamma$ is disconnected. By \cite[Lemma~2.17]{Escartin}, since $\Gamma$ has no SILs, it must be the disjoint union of two complete subgraphs. Since there cannot be infinite-order transvections there are at most one infinite labelled vertex in each connected components, thus it easily follows that $\mathcal{G}_\Gamma$ is either $\mathbb{Z}\ast\left(\mathbb{Z}/o(u)\mathbb{Z}\times\mathbb{Z}/o(v)\mathbb{Z}\right)$ or $\left(\mathbb{Z}\times\mathbb{Z}/o(u)\mathbb{Z}\right)\ast\left(\mathbb{Z}\times\mathbb{Z}/o(v)\mathbb{Z}\right)$.
	
	For $2.$,  by Lemma \ref{LemmaDisconnected} it follows that $\Gamma_v$ splits in two connected components, one containing $u$ and one having just one vertex. Call $C$ the component of $u$. There are two situations, let us first assume that there exists $u\neq w\in V(C)$ with $\{v,w\}\in E(\Gamma)$. Then $C=\st(w)\setminus\{v\}$. Hence $\lk(r)\subset\st(w)$ for each $r\in V(C)$. Since $\mathrm{Out}(\mathcal{G}_\Gamma)$ is finite that implies that $V(C)=\lbrace w,u\rbrace$. This lead us to two possible graphs, depending on whether $\lbrace u,w\rbrace\in E(\Gamma)$ or not. Now assume that there is no $u\neq w\in V(C)$ with $\{v,w\}\in E(\Gamma)$, so $\lk(v)\cap C=\{u\}$. We have two possibilities, if $\Gamma_u$ is connected then $V(C)=\{u\}$, and if $\Gamma_u$ is disconnected, by Lemma \ref{LemmaDisconnected}  we have  $\abs{V(C)}=2$. Hence, each case gives one of the four graphs of the statement.
	
	For the last statement assume that the decomposition of $\Gamma_{u,v}$ in connected components is $C_1\sqcup\dots\sqcup C_n$.  For each connected component $C_i$ we have two cases:
	\begin{itemize}
		\item If there exists $w\in V(C_i)$ such that $\{w,v\},\{w,u\}\in E(\Gamma)$ then, arguing as in Lemma \ref{LemmaDisconnected} it follows that $C_i=\{w\}$. Moreover, if there are two connected components $C_i=\{w\}$ and $C_j=\{r\}$ it follows that $\lk(w)\subset\st(r)$ contradicting that $\Out(\mathcal{G}_\Gamma)$ is finite. Hence there can only be one connected component with one vertex.
		\item If there is no $w\in V(C_i)$ as above then $\abs{V(C_i)}>1$, otherwise either $\Gamma_u$ or $\Gamma_v$ would be disconnected. Since $\Out(\mathcal{G}_\Gamma)$ is finite it follows that all partial conjugations and transvections, if any, are with respect to $u$ or $v$.
	\end{itemize}
	Therefore $\Gamma$ is a type $\mathcal{P}$ graph. By way of contradiction assume that $\mathrm{Z}(\mathcal{G}_\Gamma)\neq1$. Since $\{u,v\}$ disconnects $\Gamma$ that means that either $u\in\mathrm{Z}(\mathcal{G}_\Gamma)$ or $v\in\mathrm{Z}(\mathcal{G}_\Gamma)$. If $u\in\mathrm{Z}(\mathcal{G}_\Gamma)$ then $\st(u)=\Gamma$. Let $w\in V(C_i)$ and assume that $\abs{V(C_i)}>1$, then $\{v,w\}\notin E(\Gamma)$. This implies that $C_i$ is a connected component of $\Gamma_u$, which is a contradiction.
\end{proof}
\begin{theorem}\label{TheoremAut2} Let $\mathcal{G}_\Gamma$ be a graph product of cyclic groups with $V(\Gamma^{<\infty})=\{u,v\}$. Then: $\beta_1^{(2)}(\mathrm{Aut}(\mathcal{G}_\Gamma))\neq 0$ if and only if 
	
\begin{enumerate}
	\item Either $\mathcal{G}_\Gamma\in\{\mathbb{Z}^2\times \mathbb{Z}/o(v)\mathbb{Z}\times\mathbb{Z}/o(u)\mathbb{Z},\mathbb{Z}\ast\left(\mathbb{Z}/o(u)\mathbb{Z}\times\mathbb{Z}/o(v)\mathbb{Z}\right),\left(\mathbb{Z}\times\mathbb{Z}/o(u)\mathbb{Z}\right)\ast\left(\mathbb{Z}\times\mathbb{Z}/o(v)\mathbb{Z}\right)\}$ or
	\item $\mathcal{G}_\Gamma$ is of type $\mathcal{P}$ or
	\item $\Gamma$ is one of the following graphs
		$$\begin{tikzpicture}[main/.style = {draw, circle},node distance={15mm}] 
		\node[][main] (1)  {}; 
		\node[label={$v$}][main] (2) [right of=1] {};
		\node[][main] (3) [right of=2] {}; 
		\node[label={$u$}][main] (4) [right of=3] {};
		\draw[-] (1) -- (2);
		\draw[-] (2) -- (3);
		\draw[-] (3) -- (4);
	\end{tikzpicture}~~~~~~~~
	\begin{tikzpicture}[main/.style = {draw, circle},node distance={15mm}] 
		\node[][main] (1)  {}; 
		\node[label={$u$}][main] (2) [right of=1] {};
		\node[][main] (3) [right of=2] {}; 
		\node[label={$v$}][main] (4) [right of=3] {};
		\draw[-] (1) -- (2);
		\draw[-] (2) -- (3);
		\draw[-] (3) -- (4);
	\end{tikzpicture}$$
	$$\begin{tikzpicture}[main/.style = {draw, circle},node distance={15mm}] 
		\node[][main] (1)  {}; 
		\node[label={$v$}][main] (2) [right of=1] {};
		\node[label={$u$}][main] (3) [right of=2] {}; 
		\draw[-] (1) -- (2);
		\draw[-] (2) -- (3);
	\end{tikzpicture}~~~~~~~~\begin{tikzpicture}[main/.style = {draw, circle},node distance={15mm}] 
	\node[][main] (1)  {}; 
	\node[label={$u$}][main] (2) [right of=1] {};
	\node[label={$v$}][main] (3) [right of=2] {}; 
	\draw[-] (1) -- (2);
	\draw[-] (2) -- (3);
	\end{tikzpicture}~~~~~~~~\begin{tikzpicture}[main/.style = {draw, circle},node distance={15mm}] 
		\node[][main] (1)  {}; 
		\node[label={$v$}][main] (2) [right of=1] {};
		\node[label={$u$}][main] (3) [right of=2] {}; 
		\node[][main] (4) [right of=3] {};
		\draw[-] (1) -- (2);
		\draw[-] (2) -- (3);
		\draw[-] (3) -- (4);
	\end{tikzpicture}$$
\end{enumerate}
\end{theorem}
\begin{proof}
	By Proposition~\ref{FiniteAut}, the group $\mathrm{Aut}(\mathcal{G}_\Gamma)$ is finite if and only if $\Gamma$ is a complete graph on two or three vertices. Hence we may assume that we are not in these cases.
	
	Consider the short exact sequence
	$$1 \longrightarrow \mathcal{G}_\Gamma/\mathrm{Z}(\mathcal{G}_\Gamma)
	\longrightarrow \mathrm{Aut}(\mathcal{G}_\Gamma)
	\longrightarrow \mathrm{Out}(\mathcal{G}_\Gamma)
	\longrightarrow 1.$$
	If $\mathcal{G}_\Gamma/\mathrm{Z}(\mathcal{G}_\Gamma)$ is finite, then $\mathcal{G}_\Gamma=\mathrm{Z}(\mathcal{G}_\Gamma)$, so 
	$\mathcal{G}_\Gamma \cong \mathbb{Z}^k\times  \mathbb{Z}/o(v)\mathbb{Z}\times \mathbb{Z}/o(u)\mathbb{Z}$. 
	In this case $\mathrm{Aut}(\mathbb{Z}^k)$ embeds as a finite-index subgroup of $\mathrm{Aut}(\mathbb{Z}^k\times  \mathbb{Z}/o(v)\mathbb{Z}\times \mathbb{Z}/o(u)\mathbb{Z})$, and 
	$\beta_1^{(2)}(\mathrm{Aut}(\mathbb{Z}^k))\neq 0$ if and only if $k=2$.
	
	If both $\mathcal{G}_\Gamma/\mathrm{Z}(\mathcal{G}_\Gamma)$ and $\mathrm{Out}(\mathcal{G}_\Gamma)$ are infinite, then by \cite[Theorem~7.2~(7)]{Luck} we have $\beta_1^{(2)}(\mathrm{Aut}(\mathcal{G}_\Gamma))=0$. 
	Therefore, we may assume that $\mathrm{Out}(\mathcal{G}_\Gamma)$ is finite, so that $\mathcal{G}_\Gamma/\mathrm{Z}(\mathcal{G}_\Gamma)$ has finite index in $\mathrm{Aut}(\mathcal{G}_\Gamma)$.
	
	If $\mathrm{Z}(\mathcal{G}_\Gamma)\neq 1$, then $\mathrm{Z}(\mathcal{G}_\Gamma)\in\{\mathcal{G}_v,\mathcal{G}_u,\mathcal{G}_{u,v}\}$; otherwise, by condition~(3) of Proposition~\ref{FiniteOut}, $\mathrm{Out}(\mathcal{G}_\Gamma)$ would be infinite. Moreover, the centre being non-trivial implies that $\Gamma$ is connected.
	\begin{itemize}
		\item If $\mathrm{Z}(\mathcal{G}_\Gamma)=\mathcal{G}_v$ then $\mathcal{G}_\Gamma/\mathrm{Z}(\mathcal{G}_\Gamma)\cong \mathcal{G}_{\Gamma_v}$ and by Theorem~\ref{BettiClique} $\beta_1^{(2)}(\mathcal{G}_{\Gamma_v})=0$ if and only if $\Gamma_v$ and $\Gamma_{u,v}$ are connected. $\Gamma$ and $\Gamma_u$ are connected, so if $\Gamma_v$ is connected then by Lemma \ref{Disconnected2}.3 $\Gamma_{u,v}$ must be connected. Hence, $\beta_1^{(2)}(\Aut(\mathcal{G}_\Gamma))\neq 0$ implies $\Gamma_v$ is disconnected. Applying Lemma \ref{Gamma connected}.1 it follows that $\mathcal{G}_{\Gamma_v}\cong\mathbb{Z}\ast\mathbb{Z}/o(u)\mathbb{Z}$ and thus $\Gamma$ is the graph $\begin{tikzpicture}[main/.style = {draw, circle},node distance={15mm}] 
			\node[][main] (1)  {}; 
			\node[label={$v$}][main] (2) [right of=1] {};
			\node[label={$u$}][main] (3) [right of=2] {}; 
			\draw[-] (1) -- (2);
			\draw[-] (2) -- (3);
		\end{tikzpicture}$.
		\item If  $\mathrm{Z}(\mathcal{G}_\Gamma)=\mathcal{G}_u$ it follows in the same way as the previous case that $\begin{tikzpicture}[main/.style = {draw, circle},node distance={15mm}] 
			\node[][main] (1)  {}; 
			\node[label={$u$}][main] (2) [right of=1] {};
			\node[label={$v$}][main] (3) [right of=2] {}; 
			\draw[-] (1) -- (2);
			\draw[-] (2) -- (3);
		\end{tikzpicture}$.
		\item If $\mathrm{Z}(\mathcal{G}_\Gamma)=\mathcal{G}_{u,v}$ then  $\mathcal{G}_\Gamma/\mathrm{Z}(\mathcal{G}_\Gamma)\cong A_{\Gamma_{u,v}}$, so $\beta_1^{(2)}(\Aut(\mathcal{G}_\Gamma))\neq 0$ if and only if $\Gamma_{u,v}$ is disconnected. But $\Gamma,\Gamma_u,\Gamma_v$ connected and $\Gamma_{u,v}$ disconnected implies by Lemma \ref{LemmaDisconnected} that $\mathrm{Z}(\mathcal{G}_\Gamma)=1$, contradiction.
	\end{itemize}
	If $\mathrm{Z}(\mathcal{G}_\Gamma)=1$, then $\mathcal{G}_\Gamma$ has finite index in $\mathrm{Aut}(\mathcal{G}_\Gamma)$. 
	By Theorem~\ref{BettiClique}, $\beta_1^{(2)}(\mathcal{G}_\Gamma)=0$ if and only if $\Gamma,\Gamma_u,\Gamma_v$ and $\Gamma_{u,v}$ are connected. 
	Thus, by Lemma~\ref{Disconnected2}, we deduce that $\beta_1^{(2)}(\mathrm{Aut}(\mathcal{G}_\Gamma))=0$ unless $\mathcal{G}_\Gamma$ is one of the groups in the statement.
\end{proof}

\end{document}